%% file: main.tex
\documentclass[10pt,a4paper]{amsart}
\usepackage[margin=30mm]{geometry}
\usepackage[T1]{fontenc}
\usepackage{lmodern}
\usepackage{amsmath,amssymb,amsthm,mathtools,mathrsfs}
\usepackage{enumitem}
\usepackage[dvipsnames]{xcolor}
\usepackage{float}

\numberwithin{equation}{section}

\newtheorem{theorem}{Theorem}[section]
\newtheorem{proposition}[theorem]{Proposition}
\newtheorem{lemma}[theorem]{Lemma}
\newtheorem{corollary}[theorem]{Corollary}
\newtheorem{maintheorem}{Theorem}

\theoremstyle{definition}
\newtheorem{definition}[theorem]{Definition}
\theoremstyle{remark}
\newtheorem{remark}[theorem]{Remark}
\newtheorem{warning}[theorem]{Warning}
\usepackage[all]{xy}

\usepackage{mycom,myarrows}

\makeatletter
\def\l@subsection{\@tocline{2}{0pt}{2pc}{5pc}{}}
\makeatother

\usepackage[style=alphabetic,sorting=nyt, backend=biber, backref=true]{biblatex}%<- specify style
\usepackage[unicode, colorlinks, linktocpage=true]{hyperref}
\hypersetup{
	linkcolor = blue,
	citecolor = Red,
	urlcolor = PineGreen,
	bookmarksnumbered = true
}

\title[Zeta elements for $p$-ordinary modular forms]{Zeta elements for $p$-ordinary modular forms over imaginary quadratic fields and applications}
\usepackage{orcidlink}
\author{Ruichen Xu \,\orcidlink{0009-0003-4555-2116}}
\address[Ruichen Xu]{Faculty of Mathematics, University of Vienna. Oskar-Morgenstern-Platz 1, 1090 Vienna, Austria.}
\email{ruichen.xu@univie.ac.at}
\date{\today}

\begin{document}

\begin{abstract}
Let $p\geq 5$ be a prime, and let $f$ be a $p$-ordinary cuspidal newform of even weight $k = 2r \geq 2$ and level $N$, with $p \nmid N$ and trivial nebentype. In this article, we prove Kato's Iwasawa main conjecture for $f$, formulated in terms of the $p$-adic $L$-function of $f$, under the sole assumptions that the associated Galois representation of $f$ satisfies the big-image hypothesis and that its residual representation is absolutely irreducible. This generalises the main result of \cite{BCS}, which treats elliptic curves over $\QQ$ with good ordinary reduction at $p$. Our proof follows their general strategy. A key new ingredient is the construction, over an imaginary quadratic field in which $p$ splits, of a zeta element attached to $f$, together with a proof of its explicit reciprocity laws, generalising \cite{BSTW}. These results provide the higher-weight analogue of a crucial input in the argument of \cite{BCS} and allow their method to be carried out for higher-weight modular forms.
\end{abstract}

\maketitle

%\begin{abstract}
%Let $p\geq 5$ be a prime, and let $f$ be a $p$-ordinary cuspidal newform of even weight $k = 2r \geq 2$ and level $N$, with $p \nmid N$ and trivial nebentype. In this article, we prove Kato's Iwasawa main conjecture for $f$, formulated in terms of the $p$-adic $L$-function of $f$, under the sole assumptions that the associated Galois representation of $f$ satisfies the big-image hypothesis and that its residual representation is absolutely irreducible. This generalises the main result of Burungale--Castella--Skinner, which treats elliptic curves over $\QQ$ with good ordinary reduction at $p$. Our proof follows their general strategy. A key new ingredient is the construction, over an imaginary quadratic field in which $p$ splits, of a zeta element attached to $f$, together with a proof of its explicit reciprocity laws, generalising Burungale--Skinner--Tian--Wan. These results provide the higher-weight analogue of a crucial input in the argument of Burungale--Castella--Skinner and allow their method to be carried out for higher-weight modular forms.
%\end{abstract}

\tableofcontents

\input{01_introduction}
\input{02_setup}
\input{03_beilinson_flach}

\input{04_zeta_morphism}
\input{05_Iwasawa_main_conjectures}

%\bibliographystyle{alpha}
%\bibliography{references}

\printbibliography

\end{document}

%% file: 01_introduction.tex
\section{Introduction}\label{sec:introduction}

Let $f \in S_{k}(\Gamma_{0}(N))$ be a normalized newform of even weight $k = 2r \geq 2$, level $N$, and trivial nebentype. Let $p \geq 5$ be a prime with $p \nmid N$, and suppose that $f$ is $p$-ordinary, meaning that its $p$-th Fourier coefficient is a $p$-adic unit (under the fixed embedding $\iota_{p}: \barQQ \hookrightarrow \barQQ_p$).

Let $\QQ_{\infty}/\QQ$ be the cyclotomic $\ZZ_p$-extension of $\QQ$, and put $\Gamma := \Gal(\QQ_{\infty}/\QQ)$ with cyclotomic Iwasawa algebra $\Lambda := \calO\lrbracket{\Gamma}$, where $\calO$ is the ring of integers of a sufficiently large finite extension $L/\QQ_p$ (containing all Fourier coefficients of $f$ under $\iota_p$) with uniformizer $\varpi$.

On the algebraic side, we consider the classical ordinary Selmer group $\Sel_{\ord}(T/\QQ_{\infty})$ defined using the $p$-ordinary filtration arising from the $p$-ordinarity of $f$ on $T := T_{f}(k-1)$, where $T_{f}$ is the canonical lattice in \cite[Section 8.3]{Kato} of the Galois representation $V_{f}$ attached to $f$ by Deligne. Let $X_{\ord}(f/\QQ_{\infty})$ be the Pontryagin dual of $\Sel_{\ord}(f/\QQ_{\infty})$; it is a finitely generated $\Lambda$-module.

On the analytic side, let $\calL_p(f/\QQ) \in \Lambda[1/p]$ be the $p$-adic $L$-function attached to $f$. We use the normalization defined through the explicit reciprocity law for the Beilinson--Kato zeta element in \cite[Theorem 16.6]{Kato} (see also \cite[Theorem 2.15]{CLW}).

In this article, our main result is Kato's Iwasawa main conjecture for $f$ under mild assumptions. It generalises \cite{SkinnerUrban} (see also \cite[Theorem 1.6]{FW} and especially \cite[Section 4.8]{FW} for the case of higher-weight modular forms) by removing the extra assumptions that
\begin{equation} \tag{mult} \label{eq:mult}
    \text{there exists $\ell \nmid N$ such that $\dim_{\FF} \barT_{f}^{I_{\ell}} = 1$ and $\dim_{\FF} \barT_{f}^{G_{\ell}} = 0$.}
\end{equation}
Our result is as follows.

\begin{maintheorem}[Kato's Iwasawa main conjecture, {Theorem \ref{thm:higher-weight-cyclotomic-MC}}] \label{main:A}
Let $f \in S_{k}(\Gamma_{0}(N))$ be a normalized newform of even weight $k \geq 2$ with trivial nebentype. Suppose that $p \geq 5$, that $p \nmid N$, and that $f$ is $p$-ordinary. If, moreover, the residual Galois representation $\barV_{f}$ is absolutely irreducible, then:
\begin{enumerate}[label = \rm (\arabic*)]
    \item The Selmer group $X_{\ord}(f/\QQ_{\infty})$ is $\Lambda$-torsion.
    \item The equality $(\calL_{p}(f/\QQ)) = \chari_\Lambda X_\ordi(f/\QQ_{\infty})$ holds as ideals in $\Lambda[1/\varpi]$.
    \item If, in addition, there exists an element $\sigma \in G_{\QQ(\mu_{p^{\infty}})}$ such that 
    \begin{equation} \label{eq:big-image-intro} \tag{im}
    T/(\sigma - 1)T \simeq \calO,
    \end{equation}
    then the equality holds in $\Lambda$.
\end{enumerate}
\end{maintheorem}

We now introduce some notation in order to describe the proof. For any imaginary quadratic field $K$, denote by
\begin{itemize}
    \item $K_{\infty}/K$ (resp. $K_{\infty}^{\cyc}/K$, $K_{\infty}^{\ac}/K$) the compositum of all $\ZZ_p$-extensions (resp. cyclotomic $\ZZ_p$-extension, anticyclotomic $\ZZ_p$-extension) of $K$, and
    \item $\Lambda_{K}$ (resp. $\Lambda_{\cyc}$, $\Lambda_{\ac}$) the completed group ring $\calO \lrbracket{\Gal(K_{\infty}/K)}$ (resp. $\calO \lrbracket{\Gal(K_{\infty}^{\cyc}/K)}$, $\calO \lrbracket{\Gal(K_{\infty}^{\ac}/K)}$),
    \item $\calO^{\ur}$ the ring of integers in the completion of the maximal unramified extension of $L$, with $R^{\ur} := R \hotimes_{\calO} \calO^{\ur}$ when $R$ is one of the Iwasawa algebras $\Lambda_{K}$ and $\Lambda_{\ac}$;
    \item for an Iwasawa algebra $R$, we write $R^{\iota}$ for $R$ endowed with the $G_{K}$-action given by the inverse of the tautological character (see Section \ref{sec:Iwasawa_algebras}).
\end{itemize}

This article generalizes \cite{BCS}: the work \cite{BCS} treats the case where $f$ is the normalized newform attached, via modularity, to an elliptic curve $E/\QQ$ with good ordinary reduction at $p \geq 5$. We follow the same methodology as in \cite{BCS}:
\begin{enumerate}
    \item The first key idea is a \emph{base-change} method: one constructs an appropriate real quadratic field $F/\QQ$ and an imaginary quadratic field $K/\QQ$ (with $p$ split in $K$ as $p\calO_{K} = \frp \barfrp$) such that \emph{Wan's three-variable Iwasawa main conjecture} \cite{Wan2015} applies to the base change of $f$ to $F$ over the CM extension $M := FK/F$. It then descends to a product of two ordinary two-variable Iwasawa main conjectures (see \eqref{eq:PRMC}) for $f$ and its quadratic twist $f^{F}$ over $K$, involving Perrin--Riou's $p$-adic $L$-function $\calL_{p}^{\PR}(f/K)$. The key drawback is that, after descending to $K$, these Iwasawa main conjectures are valid only in the Iwasawa algebra $\Lambda_{K}[1/\varpi]$, i.e. with $\varpi$ inverted.
    \item The second key idea is an \emph{equivalence of Iwasawa main conjectures}: the ordinary Iwasawa main conjecture for $f$ over $K$ is equivalent to Greenberg's Iwasawa main conjecture (see \eqref{eq:GIMC}) for $f$ over $K$. After passing to Greenberg's Iwasawa main conjecture, one can use the vanishing of the $\mu$-invariant of Greenberg's $p$-adic $L$-function $\calL_{p}^{\Gr}(f/K)$ proved by Hsieh \cite{Hsieh2014}. This yields an integral result over $\Lambda_{K}^{\ur}$ without inverting $\varpi$. Transferring back to the ordinary Iwasawa main conjecture then proves the main theorem.
\end{enumerate}

Although (1) carries over to our setting from \cite{BCS}, the equivalence in (2) between the ordinary and Greenberg Iwasawa main conjectures is not available in the literature. For elliptic curves $E/\QQ$ with good ordinary reduction, it is proved in \cite[Proposition 9.18]{BSTW} using the zeta element $\calZ(E/K)$ attached to $E$ over $K$, which encodes both $\calL_{p}^{\PR}(E/K)$ and $\calL_{p}^{\Gr}(E/K)$ through its explicit reciprocity laws at $\frp$ and $\barfrp$, respectively. To carry out the arguments of \cite{BCS}, we generalize the corresponding results of \cite{BSTW} to higher-weight, good $p$-ordinary normalized cuspidal newforms $f$, including the construction of the zeta element $\calZ(f/K)$ and its two explicit reciprocity laws. We summarize the result as follows.

\begin{maintheorem} \label{main:C}
    With the notation above, let $K$ be an imaginary quadratic field over $\QQ$ such that $(D_{K}, N) = 1$, such that $p$ splits in $K$, and such that $\barV_{f}|_{G_{K}}$ is irreducible. Then there exists a zeta element
    \[
    \calZ(f/K) \in \rmH^{1}_{\rel, \ord}(\calO_{K}[1/p], T \otimes \Lambda_{K}^{\iota})
    \]
    such that 
    \[
    \Cole(f/K)(\loc_{\frp}(\calZ(f/K))) = \calL_{p}^{\PR}(f/K), \quad \Log(f/K)(\loc_{\barfrp}(\calZ(f/K))) = \calL_{p}^{\Gr}(f/K),
    \]
    where $\loc_{\frp}$ and $\loc_{\barfrp}$ are localisation maps, and $\Cole(f/K): \rmH^{1}(K_{\frp}, T \otimes \Lambda_{K}^{\iota}) \rightarrow \Lambda_{K}$ and $\Log(f/K): \rmH^{1}_{\ord}(K_{\barfrp}, T \otimes \Lambda_{K}^{\iota}) \otimes_{\calO} \calO^{\ur} \rightarrow \Lambda_{K}^{\ur}$ are Perrin--Riou regulator maps interpolating Bloch--Kato dual exponential and logarithm maps, respectively, as described in Section \ref{sec:zeta_definition}. The resulting functions $\calL_{p}^{\PR}(f/K)$ and $\calL_{p}^{\Gr}(f/K)$ satisfy
    \begin{enumerate}[label = \rm (\arabic*)]
        \item $(\calL^{\PR}_p(f/K)^{\eta} \pmod{I_{\ac}}) = (\calL_{p}(f/\QQ)^{\eta} \cdot \calL_{p}(f^{K}/\QQ)^{\eta})$, where $I_{\ac} := \ker(\Lambda_{K} \twoheadrightarrow \Lambda_{\cyc})$, and
        \item under the generalized Heegner hypothesis required for $(f,K)$ (see \eqref{eq:gen-Heegner}), we have 
        \[
        (\calL^{\Gr}_p(f/K)^{\omega^{1-r}} \pmod{I_{\cyc}^{\ur}}) = (\calL_{p}^{\BDP}(f/K)),
        \]
        where $I_{\cyc}^{\ur} := \ker(\Lambda_{K}^{\ur} \twoheadrightarrow \Lambda_{\ac}^{\ur})$.
        \end{enumerate}
    Here $\calL_p(f/\QQ) \in \Lambda$ is the cyclotomic $p$-adic $L$-function of $f$, and $\calL_{p}^{\BDP}(f/K) \in \Lambda_{\ac}^{\ur}$ is the Bertolini--Darmon--Prasanna $p$-adic $L$-function for $f$ over $K$.
\end{maintheorem}

The proof of Theorem \ref{main:C} follows the same strategy as \cite{BSTW}, particularly in \cite[Section 5]{BSTW} where elliptic curves $E/\QQ$ with good ordinary reduction at $p$ are treated. The key is to construct $\calZ(f/K)$ and obtain its explicit reciprocity laws from the Beilinson--Flach element $\cBF^{\bff, \bfg}$ of \cite{KLZ} applied with $\bff$ passing through the ordinary $p$-stabilization of $f$ and with $\bfg$ the canonical CM family constructed in \cite[Section 4]{BSTW}. The inputs from \cite{KLZ} and \cite{Kato} apply when $f$ is a $p$-ordinary normalized cuspidal newform of weight $k \geq 2$ with $p \nmid N$, so the same method carries over.

To conclude the introduction, we make the following three remarks.

(1) In this article, we treat the $p$-ordinary case in higher weight. The higher-weight $p$-supersingular counterpart of \cite{BSTW} has been addressed in the recent work \cite{CLW}, with applications to Kato's Iwasawa main conjectures and Tamagawa number conjectures in cases of analytic rank zero.

(2) In \cite[Theorem 1.2.2]{BCS}, the authors prove the \emph{Heegner point main conjecture} for elliptic curves $E/\QQ$ with good ordinary reduction at $p$. More precisely, they establish the equality of ideals in $\Lambda_{\ac}$
\[
\chari_{\Lambda_{\ac}}\bigl(X_{\ord}(E/K_{\infty}^{\ac})\bigr)
=
\chari_{\Lambda_{\ac}}\bigl(
\rmH^{1}_{\ord}(K,\bfT)/
(\boldsymbol{\kappa}_{1}^{\mathrm{Heeg}})
\bigr)^{2}.
\]
Here $\bfT:=\varprojlim_{n}\Ind_{K_n^{\ac}/K}(T)$ and $\rmH^{1}_{\ord}(K,\bfT)\subset \rmH^{1}(K,\bfT)$ denotes the compact ordinary Selmer group interpolating, as $n$ varies, the classical Selmer groups $\varprojlim_m \Sel_{p^m}(E/K_n^{\ac})$. The class $\boldsymbol{\kappa}_{1}^{\mathrm{Heeg}}
\in
\rmH^{1}_{\ord}(K,\bfT)$ is the corresponding $\Lambda_{K}^{\ac}$-adic Heegner point.

It is natural to expect that an analogous result should hold for higher-weight modular forms in the setting of the present paper. Indeed, combined with Theorem \ref{main:C}, the same base-change argument as in the weight-two case should reduce such a generalization to the corresponding Kolyvagin-system-side divisibility in the Heegner cycle main conjecture. \footnote{In the setting of Kato's Iwasawa main conjecture, the required divisibility is provided by the Euler system of Beilinson--Kato elements, as recalled in Theorem \ref{thm:kato_17_4}.} However, the analogous divisibility arising from Kolyvagin systems of Heegner cycles is not yet available in the generality required for this application.

(3) In Theorem \ref{main:A}, to obtain the equality in the Iwasawa main conjecture over $\Lambda$ without inverting $\varpi$, we assume the classical "big image hypothesis" \eqref{eq:big-image-intro}. This excludes some modular forms with good ordinary reduction at $p$. In forthcoming joint work \cite{FGX} with Olivier Fouquet and Arthur Gérard, the author replaces this condition by the vanishing of certain fine $\mu$-invariants, together with other technical assumptions; this vanishing of fine $\mu$-invariant, unlike the big-image assumption \eqref{eq:big-image-intro}, is expected for all modular forms. That method requires arguments parallel to those in Section \ref{sec:application_cmc}, which depend on Theorem \ref{thm:BCS-521-higher} established here. The latter, in turn, depends on Theorem \ref{main:C} in this article. This was the original motivation for the present work.

\subsection*{Acknowledgements}

We thank Professor Xin Wan for suggesting this problem to us, for bringing his recent work \cite{CLW} to our attention, and for many helpful discussions and valuable guidance throughout the course of this work. 

This work was carried out during the author's postdoctoral appointment at the University of Vienna under the supervision of Professor Harald Grobner and the co-supervision of Professor Adel Betina, to whom we would like to express our sincere gratitude for their support and guidance. This work (together with the author's postdoctoral position) is supported by the \emph{FWF Principal Investigator Project (PAT-2584625): Eigenvarieties and $p$-adic $L$-functions}.

%% file: 02_setup.tex
\section{Setup}
\label{sec:setup}

In this section, we introduce the basic setup, including Iwasawa algebras, Selmer groups, and the canonical CM family of \cite[Section 4]{BSTW}.

Throughout, we fix an algebraic closure $\barQQ$ of $\QQ$ and embeddings $\iota: \barQQ \hookrightarrow \CC$ and $\iota_{p}: \barQQ \hookrightarrow \CC_p$.

\subsection{Iwasawa algebras} \label{sec:Iwasawa_algebras}
Fix an odd prime $p$, and put
\[
\tilGamma := \Gal(\QQ(\mu_{p^\infty})/\QQ) \simeq \Delta \times \Gamma
\]
where $\Gamma$ is the Galois group of the cyclotomic $\ZZ_p$-extension $\QQ_{\infty}/\QQ$, and $\Delta$ is a cyclic group of order $p-1$. We let
\[
\Lambda := \calO\lrbracket{\Gamma}, \quad \tilLambda := \calO\lrbracket{\tilGamma}
\]
be the cyclotomic Iwasawa algebras. We often fix a character $\eta: \Delta \rightarrow \calO^{\times}$ of $\Delta$.

Let $K$ be any imaginary quadratic field over $\QQ$. Let $K_\infty/K$ be the compositum of all $\ZZ_p$-extensions of $K$, and denote
\[
\Gamma_K := \Gal(K_\infty/K) \simeq\ZZ_p^2.
\]
The Galois groups of its cyclotomic $\ZZ_p$-extension $K_{\infty}^{\cyc}/K$ and its anticyclotomic $\ZZ_p$-extension $K_{\infty}^{\ac}/K$ are denoted by $\Gamma_\cyc$ and $\Gamma_\ac$, respectively. Let $L/\QQ_p$ be a finite extension with ring of integers $\calO$. Put
\[
\Lambda_K=\calO\lrbracket{\Gamma_{K}},\qquad \Lambda_\cyc=\calO\lrbracket{\Gamma_{\cyc}},\qquad \Lambda_\ac=\calO\lrbracket{\Gamma_{\ac}}.
\]
We often identify $\Lambda$ and $\Lambda_{\cyc}$ without further comment.

For any of these rings $R$, the notation $R^\iota$ denotes $R$ with the $G_K$-action given by the inverse of the tautological character $\Psi_{\dagger}: G_{K} \twoheadrightarrow \Gamma_{\dagger} \hookrightarrow R^{\times}$ for $\dagger \in \{ K, \cyc, \ac\}$. Every such $R$ is a complete regular local domain. 

Let $\kappa_0$ be the residue field of the maximal unramified subextension of $L/\QQ_p$, fix an embedding $\kappa_0\hookrightarrow\overline{\FF_p}$, and put
\[
\calO^{\ur}=\calO\hotimes_{W(\kappa_0)} W(\overline{\FF_p}), \qquad R^{\ur}=R\hotimes_{\calO}\calO^{\ur}.
\]
Thus $\calO^{\ur}$ is the ring of integers in the completion of the maximal unramified extension of $L$ selected by the embedding.  In particular, $R^{\ur}$ is again a complete regular local domain and is faithfully flat over $R$. \footnote{Since $\calO^{\ur}$ is torsion-free over the discrete valuation ring $\calO$, it is a flat $\calO$-algebra. Therefore $R \rightarrow R^{\ur}$ is flat as well. Moreover, the map $R \rightarrow R^{\ur}$ is a local homomorphism. Any flat local homomorphism is faithfully flat, see \cite[\href{https://stacks.math.columbia.edu/tag/00HR}{Tag 00HR}]{stacks-project}. Hence $R \rightarrow R^{\ur}$ is faithfully flat.}  If $M$ is an $R$-module, write $M^{\ur}=M\hotimes_RR^{\ur}$.

\subsection{Galois representations and Selmer conditions}

Let $L/\QQ_p$ be a finite extension, and let $V$ be a finite-dimensional $L$-representation of $G_{K}$. Choose a $G_K$-stable lattice $T\subset V$ and put $W=V/T$. We assume that $S$ contains the archimedean places, the places above $p$, and every place at which $T$ is ramified. Suppose that $T$ is \emph{$p$-ordinary}, in the sense that it is equipped with a unique $G_{\QQ_p}$-stable short exact sequence
\[
0 \rightarrow \Fil^{+} T \rightarrow T \rightarrow \Fil^{-} T \rightarrow 0,
\]
in which both $\Fil^{+} T$ and $\Fil^{-} T$ are free $\calO$-modules and $\Fil^{-} T$ is unramified. We use the same notation after restriction to $G_{K_{\frp}}$ or $G_{K_{\bar{\frp}}}$ and after scalar extension.

For each ring $R$ among $\Lambda_{K}$, $\Lambda_{\cyc}$, and $\Lambda_{\ac}$, we set
\[
\bfT_{R} := T \hotimes_{\calO} R^{\iota}, \quad \bfA_{R} := \Hom_{\cts}(\bfT_{R}, \QQ_p/\ZZ_p)(1)
\]
together with $\Fil^{\pm} \bfT_{R} := \Fil^{\pm} T \hotimes_{\calO} R^{\iota}$. With a character $\eta: \Delta \rightarrow \calO^{\times}$ fixed, we shall always understand $T$ as a twist of it by $\eta$. The local Tate duality relates $\bfT_{R}$ and $\bfA_{R}$ by 
\[
\rmH^{1}(K_w, \bfT_{R}) \times \rmH^{1}(K_w, \bfA_{R}) \longrightarrow \QQ_p/\ZZ_p.
\]

Fix a finite set $S$ containing the places above $p$, the archimedean places, and all places at which $T$ is ramified. Write $\rmH^1(K,\bfT_{R}):=\rmH^1(G_{K,S},\bfT_{R})$. A \emph{compact Selmer structure} $\calF$ on $\bfT_{R}$ consists of submodules
\[
\rmH^1_{\calF}(K_w,\bfT_{R}) \subseteq H^1(K_w,\bfT_{R})
\]
for every finite place $w$ of $K$. Its \emph{compact Selmer group} is defined as
\[
\rmH^1_{\calF}(K,\bfT_{R}) := \ker\left( \rmH^1(G_{K,S},\bfT_{R}) \rightarrow \bigoplus_{w \in S} \dfrac{\rmH^1(K_w, \bfT_{R})}{\rmH^1_{\calF}(K_w, \bfT_{R})}\right),
\]
and the \emph{dual discrete Selmer structure} is $\rmH^1_{\calF^*}(K_w, \bfA_{R}) := \rmH^1_{\calF}(K_w, \bfT_{R})^\perp$ under local Tate duality. The corresponding \emph{dual discrete Selmer group} is
\[
\Sel_{\calF}(K,\bfA_{R}) := \ker\left(\rmH^1(G_{K,S},\bfA_{R}) \rightarrow \bigoplus_{w \in S} \dfrac{\rmH^1(K_w, \bfA_{R})}{\rmH^1_{\calF^*}(K_w, \bfA_{R})}\right),
\]
and its Pontryagin dual is denoted by $X_{\calF}(K, \bfA_{R})$.

We define the following local Galois cohomology groups as compact Selmer structures:
\[
\rmH^{1}_{\bullet}(K_w, \bfT_{R}) := \begin{cases}
\ker\bigl(\rmH^1(K_w,\bfT_R)\rightarrow \rmH^1(K_w,\Fil^{-}\bfT_R)\bigr) , &\quad w \mid p, \quad \bullet = \ord ; \\
\rmH^{1}(K_w, \bfT_{R}), &\quad w \mid p, \quad \bullet = \rel ; \\
0, &\quad w \mid p, \quad \bullet = \str ; \\
\ker(\rmH^{1}(K_w, \bfT_{R}) \rightarrow \rmH^{1}(I_w, \bfT_{R})), &\quad w \nmid p, \quad \bullet = \ur,
\end{cases}
\]
called the \emph{ordinary}, \emph{relaxed}, \emph{strict}, and \emph{unramified} local conditions, respectively.

We need the following three compact Selmer structures:
\begin{center}
\begin{tabular}{c|ccc}
& at $\frp$ & at $\barfrp$ & at $v \nmid p$, $v \in S$ \\ \hline
$\calF_{\ordi,\ordi}$ & ordinary & ordinary & unramified \\
$\calF_{\rel,\ordi}$ & relaxed & ordinary & unramified \\
$\calF_{\rel,\str}$ & relaxed & strict & unramified.
\end{tabular}
\end{center}

\subsection{Galois representation of modular forms and $p$-ordinary filtration} \label{sec:modforms}
We follow the convention of \cite[Section~2.1]{CLW}. Let
\[
f=\sum_{n\geq 1}a_n(f)q^n\in S_k(\Gamma_0(N)),\qquad k=2r\geq 2,
\]
be a normalized newform of trivial nebentype. Fix a prime $p \geq 3$ with $p\nmid N$.

Let $F = \QQ(\{a_n\}_n)$ be the Hecke field of $f$, viewed inside $\barQQ$ via $\iota_{\infty}$, and fix $L/\QQ_p$ containing the image of the Hecke field and both roots of the Hecke polynomial at $p$. We assume that $f$ is \emph{$p$-ordinary}, meaning that $\iota_p(a_p)$ is a $p$-adic unit. Hence its Hecke polynomial
\[
X^2-a_p(f)X+p^{k-1}=(X-\alpha)(X-\beta)
\]
has a unique root $\alpha\in\calO^\times$, which is a $p$-adic unit, while the other root is $\beta=p^{k-1}/\alpha$.

Let $K/\QQ$ be an imaginary quadratic field with discriminant $-D_K$ such that
\begin{equation} \label{eq:spl} \tag{spl}
(D_K,N)=1,\qquad p\calO_K=\frp\barfrp.
\end{equation}
The prime $\frp$ is the one selected by the fixed $p$-adic embedding $\iota_p$.

Let $V_f$ be Deligne's $E$-linear representation with Hodge--Tate weights $0$ and $1-k$, where the cyclotomic character has Hodge--Tate weight $+1$.  Let $T_f\subset V_f$ be Kato's canonical lattice \cite[Section~8.3]{Kato}, and set
\[
V=V_f(k-1),\qquad T=T_f(k-1),\qquad \overline T=T/\varpi T.
\]
Thus the Hodge--Tate weights of $V$ are $0$ and $k-1$. Ordinarity of $f$ supplies a unique $G_{\QQ_p}$-stable exact sequence
\begin{equation}\label{eq:ordinary-filtration}
 0\rightarrow \Fil^+T\rightarrow T\rightarrow \Fil^-T \rightarrow 0
\end{equation}
in which $\Fil^+T$ and $\Fil^-T$ are free of rank one over $\calO$, and $\Fil^-T$ is unramified.

Within Sections \ref{sec:setup}--\ref{sec:zeta_main_conjecture}, we always assume that
\begin{equation} \tag{irred$_{K}$} \label{eq:irred}
    \text{The residual $\Gal_{K}$-representation $\barV$ is absolutely irreducible}.
\end{equation}
This assumption is required to construct the zeta element $\calZ(f/K)$ for $f$ over $K$, as well as its reciprocity laws and further properties.

\subsection{Property of global Iwasawa cohomology groups}

Kato's Euler system and its explicit reciprocity law give the following higher-weight counterpart of \cite[Theorems~3.7--3.11]{BSTW}.

\begin{proposition}\label{prop:rank}
Suppose that \eqref{eq:irred} holds. Then the following statements hold.
\begin{enumerate}[label = \rm (\arabic*)]
    \item $\rmH^1(K,\bfT_{\cyc})$ is a free $\Lambda_{\cyc}$-module of rank $2$. 
    \item $\rmH^1_{\rel,\ordi}(K,\bfT_{\cyc})$ is a free $\Lambda_{\cyc}$-module of rank one.
    \item If \eqref{eq:irred} holds, then Kato's zeta element $\bfz_{\alpha, \omega, \gamma, \gamma^{\prime}}(f/K) \in \rmH^{1}_{\rel,\ordi}(\calO_{K}[1/p], T \otimes_{\calO} \Lambda_{\cyc})$ (see \cite[Section 3.2.3]{BSTW}) is not $\Lambda_{\cyc}$-torsion.
\end{enumerate}
\end{proposition}

\begin{proof}
    Since \cite{Kato} applies to all cuspidal eigenforms of weight $k \geq 2$, the arguments of \cite[Theorems~3.7--3.11]{BSTW} carry over directly. Part (1) is \cite[Theorem 3.9]{BSTW} (which follows from \cite[Theorem 3.7]{BSTW}), and part (2) is \cite[Theorem 3.11]{BSTW}. Note that we assume \eqref{eq:irred} throughout this article. The last statement is \cite[Theorem 3.12]{BSTW}.
\end{proof}

In this article, the parallel result for $R = \Lambda_{K}$ will be proved in Proposition \ref{prop:rank-one-nonvanishing}.

\subsection{The canonical CM family and its Galois representation} \label{sec:CM}
We follow the convention of \cite[Section~2.1]{CLW} originating in \cite[Section 4]{BSTW}. 

\subsubsection{More Iwasawa algebras} \label{sec:more-iwasawa-algebra}
Recall that we have assumed \eqref{eq:spl}. Let $\Gamma_\frp\simeq\ZZ_p$ be the Galois group of the $\ZZ_p$-extension of $K$ ramified only at $\frp$, and put $\Lambda_\frp=\ZZ_p\lrbracket{\Gamma_{\frp}}$. As usual, we choose a topological generator $\gamma_\frp$ and put $ T_{\frp}=\gamma_\frp-1$, so $\Lambda_{\frp}$ is identified with the power-series ring $\ZZ_p\lrbracket{T_{\frp}}$. We denote the tautological character by $\Psi_{\frp}: G_{K} \twoheadrightarrow \Gamma_{\frp} \hookrightarrow \Lambda_{\frp}^{\times}$.  Let $\rec_{K}: \AA_{K}^{\times} \rightarrow G_{K}^{\ab}$ be the geometrically normalised reciprocity map of class field theory, and let $\rec_{\frp}: K_{\frp}^{\times} \rightarrow G_{K_{\frp}}^{\ab}$ denote the restriction of $\rec_{K}$ to $K_{\frp}^{\times}$. Then $\rec_{\frp}(1+p\ZZ_p)$ injects into $\Gamma_{\frp}$ under the identification between $1+p\ZZ_p$ and the group of $1$-units in $\calO_{K_{\frp}}$, and hence $[\Gamma_{\frp}: \rec_{\frp}(1+p\ZZ_p)] = p^{h_{p}}$ for some $h_{p} \geq 0$ (with $h_{p} = 0$ if the class number $h_{K}$ of $K$ is coprime to $p$).

The Galois group $\Gal(K/\QQ) = \{1,c\}$ acts on $\Gamma_{K}$ via conjugation by any lift of $c$ to $\Gal(K_{\infty}/\QQ)$, and
\[
\Gamma_{+} := \Gal(K_{\infty}/K_{\ac}), \quad \Gamma_{-} := \Gal(K_{\infty}/K_{\cyc})
\]
are the $\ZZ_p$-summands of $\Gamma_{K}$ on which this action is multiplication by $\pm 1$, and $\Gamma_{+}$ (resp. $\Gamma_{-}$) maps isomorphically onto $\Gamma_{\cyc}$ (resp. $\Gamma_{\ac}$) under the canonical projection from $\Gamma_{K}$. 

In light of the canonical isomorphisms $\Gamma_{+} \xrightarrow{\sim} \Gamma_{\cyc} \xrightarrow{\sim} \Gamma$ and $\Gamma_{-} \xrightarrow{\sim} \Gamma_{\ac}$, we can and do choose topological generators $\gamma_{\square} \in \Gamma_{\square}$ for $\square \in \{ \pm, \cyc, \ac, \frp\}$ such that:
\begin{itemize}
    \item $\gamma_{+}$ and $\gamma_{\cyc}$ are identified with the previously chosen $\gamma \in \Gamma$,
    \item $\gamma_{-}$ is identified with $\gamma_{\ac}$,
    \item $c \gamma_{\frp} c^{-1} = \gamma_{\barfrp}$,
    \item $\gamma_{+}$ maps to $\gamma_{w}^{p^{h_{p}}/2}$ for $w \in \{\frp, \barfrp \}$,
    \item $\gamma_{-}$ maps to $\gamma_{\frp}^{1/2}$ and to $\gamma_{\barfrp}^{-1/2}$.
\end{itemize}
Under these choices, we have the identification $\iota_{\ac}: \Lambda_{\ac}^{\ur} \xrightarrow{\sim} \Lambda_{\frp}^{\ur}$ sending $\gamma_{\ac}$ to $\gamma_{\frp}^{1/2}$, and its inverse $\iota_{\frp}: \Lambda_{\frp}^{\ur} \xrightarrow{\sim} \Lambda_{\ac}^{\ur}$ given by $\gamma_{\frp} \mapsto \gamma_{\ac}^{2}$.

In addition, following \cite[Equation (5.16)]{BSTW}, we define the following composition of isomorphisms:
\[
\theta: \Gamma_{K} \xrightarrow{\sim} \Gamma_{\frp} \times \Gamma_{\cyc}, g \mapsto (\pr_{\frp}(g)^{-1}, \, \pr_{\cyc}(g)),
\]
where $\pr_{\square}: \Gamma_{K} \twoheadrightarrow \Gamma_{\square}$ for $\square \in \{\frp, \cyc\}$ are natural projections.

\subsubsection{The canonical CM family and Hecke algebras}
Let $\bfh_\frp \in \Lambda_{\frp}\lrbracket{q}$ be the canonical ordinary CM Hida family of tame level $D_K$ in \cite[Section~3.1]{CLW}. We put $\Lambda_D := \ZZ_p\lrbracket{\ZZ_p^{\times}}$, and let $M$ be an integer prime to $p$. Following \cite[Sections 2.2.2 and 4.2.2]{BSTW}, let
\[
\frh_{Mp^{\infty}}^{\ord} := \varprojlim_{r} e_{\ord}^{\prime} \frh_{\Gamma_1(Mp^{r})}, \quad \frH_{Mp^{\infty}}^{\ord} := \varprojlim_{r} e_{\ord}^{\prime} \frH_{\Gamma_1(Mp^{r})},
\]
be Hida's ordinary Hecke algebras acting on $\Lambda_{D}$-adic cusp forms and modular forms, respectively, and note that $\frH_{Mp^{\infty}}^{\ord}$ surjects onto $\frh_{Mp^{\infty}}^{\ord}$ by restriction.

Put
\begin{equation}\label{eq:big-H}
\begin{aligned}
\rmH^1_{\ord}(Mp^\infty)&:=\varprojlim_re_{\ord}'\rmH^1_{\et}(X_1(Mp^r)_{\overline{\QQ}},\ZZ_p(1)),\\
\widetilde{H}^1_{\ord}(Mp^\infty)&:=\varprojlim_re_{\ord}'\rmH^1_{\et}(Y_1(Mp^r)_{\overline{\QQ}},\ZZ_p(1)),
\end{aligned}
\end{equation}
equipped with their natural $\frh_{Mp^\infty}^{\ord}[G_\QQ]$- and $\frH_{Mp^\infty}^{\ord}[G_\QQ]$-module structures, respectively.

Let $\varphi_v:\frh_{D_Kp^\infty}^{\ord}\rightarrow\Lambda_{\frp}$ be the $\Lambda_D$-algebra homomorphism corresponding to $\bfh_{\frp}$, where $\Gamma_{\frp}$ is viewed as a $\Lambda_D$-module via the $\ZZ_p$-algebra homomorphism $\Lambda_D\rightarrow\Gamma_{\frp}$ determined by $[\langle\varepsilon(\tilde{\gamma})\rangle]\mapsto\varepsilon(\tilde{\gamma})^{-1}\gamma_v$ for all $\tilde{\gamma}\in G_K$ via the cyclotomic character $\varepsilon:G_K\rightarrow\ZZ_p^\times$, and put
\begin{equation}\label{def:CM-Tate-lattice}
\begin{aligned}
\TT_{\bfh_v}&:=H^1_{\ord}(D_{K} p^\infty)\otimes_{\frh_{D_{K}p^\infty}^{\ord}}\Gamma_{\frp},\\ 
\widetilde{\TT}_{\bfh_v}&:=\widetilde{H}^1_{\ord}(D_{K}p^\infty)\otimes_{\frH_{D_{K}p^\infty}^{\ord}}\Gamma_{\frp},
\end{aligned}
\end{equation}
where the tensor products are with respect to $\varphi_v$ and its composition with the projection $\frH_{D_Kp^\infty}^{\ord}\rightarrow\frh_{D_Kp^\infty}^{\ord}$, respectively.

We define $\TT_{\bfh_\frp}^\square$ (resp. $\tilTT_{\bfh_\frp}^\square$) to be the quotient of $\TT_{\bfh_\frp}$ (resp. $\tilTT_{\bfh_\frp}$) by its torsion submodule. 

\subsubsection{Local properties at $p$}
We recall Ohta's filtration on ordinary cohomology. By Ohta's work \cite{OhtaII}, the modules $H^1_{\rm ord}(Mp^\infty)$ and $\widetilde{H}^1_{\rm ord}(Mp^\infty)$ introduced in Section \ref{sec:CM} are equipped with $\mathfrak{H}_{Mp^\infty}^{\rm ord}[G_{\QQ_p}]$-stable filtrations 
\begin{align*}
\Fil^{+} \rmH^1_{\ord}(Mp^\infty)&\subset \rmH^1_{\rm ord}(Mp^\infty),\\
\Fil^{+}\widetilde{\rmH}^1_{\rm ord}(Mp^\infty)&\subset \widetilde{\rmH}^1_{\ord}(Mp^\infty)
\end{align*}
where $\Fil^{+}H^1_{\rm ord}(Mp^\infty)=\Fil^{+}\widetilde{H}^1_{\rm ord}(Mp^\infty)$ is isomorphic to $\frh_{Mp^\infty}^{\rm ord}$ as a $\mathfrak{H}^{\rm ord}_{Mp^\infty}$-module. Put 
\begin{align*}
\Fil^{+} \TT_{\bfh_\frp}&:=\Fil^{+}H^1_{\rm ord}(Mp^\infty)\otimes_{\frh_{M p^\infty}^{\rm ord}}\Lambda_{\ZZ_p}(\Gamma^v),\\
\Fil^{+} \widetilde{\TT}_{\bfh_\frp}&:=\Fil^{+}\widetilde{H}^1_{\rm ord}(Mp^\infty)\otimes_{\frh_{M p^\infty}^{\rm ord}}\Lambda_{\ZZ_p}(\Gamma^v),
\end{align*}
so we have a commutative diagram with exact rows
\begin{equation}\label{eq:fil-Ohta}
\begin{tikzcd}[column sep=large, row sep=large]
0
\arrow[r]
&
\Fil^{+}\TT_{\bfh_\frp}
\arrow[r]
\arrow[d, equal]
&
\TT_{\bfh_\frp}^{\square}
\arrow[r]
\arrow[d, hook]
&
\Fil^{-} \TT_{\bfh_\frp} :=
\TT_{\bfh_\frp}^{\square}/\Fil^{+}\TT_{\bfh_\frp}
\arrow[r]
\arrow[d, hook]
&
0
\\
0
\arrow[r]
&
\Fil^{+}\widetilde{\TT}_{\bfh_\frp}
\arrow[r]
&
\widetilde{\TT}_{\bfh_\frp}^{\square}
\arrow[r]
&
\Fil^{-} \widetilde{\TT}_{\bfh_\frp} := \widetilde{\TT}_{\bfh_\frp}^{\square}/
\Fil^{+}\widetilde{\TT}_{\bfh_\frp}
\arrow[r]
&
0
\end{tikzcd}.
\end{equation}

\subsubsection{Various CM lattices}
Following \cite[Section 4.2.5]{BSTW}, for any finite $\Lambda_{\frp}$-module $M$, the \emph{reflexive closure} of $M$ is defined as
\[
M^{\star} := \bigcap_{\wp \in \Spec_1(\Lambda_{\frp})} (M/M_{\tor})_{\wp} \subseteq M \otimes_{\Lambda_{\frp}} \Frac(\Lambda_{\frp}), 
\]
where $\Spec_{1}(-)$ is the set of height-one prime ideals. Since $\Lambda_{\frp}$ is a two-dimensional regular local ring, the reflexive closure $M^{\ast}$ is a free $\Lambda_{\frp}$-module. We note that $\TT_{\bfh_{\frp}}^{\star} = (\TT_{\bfh_{\frp}}^{\square})^{\star}$ and the same holds for $\tilTT$ replacing $\TT$.

Following \cite[Section 4.2.6]{BSTW}, we define
\[
\TT_{\bfh_{\frp}, \frp} := ( \Fil^{+} \TT_{\bfh_{\frp}} \otimes_{\Lambda_{\frp}} \Frac(\Lambda_{\frp})) \cap (\TT_{\bfh_{\frp}})^{\star}, \quad \tilTT_{\bfh_{\frp}, \frp} := ( \Fil^{+} \tilTT_{\bfh_{\frp}} \otimes_{\Lambda_{\frp}} \Frac(\Lambda_{\frp})) \cap (\tilTT_{\bfh_{\frp}})^{\star},
\]
and write
\[
\TT_{\bfh_{\frp}, \barfrp} := c \cdot \TT_{\bfh_{\frp}, \frp} \quad \tilTT_{\bfh_{\frp}, \barfrp} := c \cdot \tilTT_{\bfh_{\frp}, \frp},
\]
where $c$ denotes complex conjugation.

We collect the main results of \cite[Section 4]{BSTW} in the following theorem.

\begin{theorem} \label{thm:CM-lattices}
    With the notation above, the following statements hold:
    \begin{enumerate}[label = \rm (\arabic*)]
        \item The element $T_{\frp}$ annihilates the quotient $(\tilTT_{\bfh_\frp})^{\star}/(\TT_{\bfh_\frp})^{\star}$.
        \item We have
        \[
        (\TT_{\bfh_\frp})^{\star} \subseteq \TT_{\bfh_\frp}^\square \subseteq (\tilTT_{\bfh_\frp})^{\star}
        \]
        with $\TT_{\bfh_\frp}^{\square} \simeq {\rm Ind}_K^{\QQ}\,\Lambda_{\frp}^\iota$ as $\Gamma_{\frp}[G_\QQ]$-modules. In particular, we have $\TT_{\bfh_\frp}^{\square} \simeq \TT_{\bfh_\frp, \frp} \oplus c \cdot \TT_{\bfh_\frp, \frp}$, with $\TT_{\bfh_{\frp},\frp}$ isomorphic to $\Gamma_{\frp}^{\iota}$ as a $\Lambda_{\frp}[G_{K}]$-module.
    \end{enumerate}
\end{theorem}
\begin{proof}
Part (1) follows from \cite[Proposition 4.13]{BSTW}, while part (2) combines \cite[Lemma 4.18]{BSTW} and \cite[Theorem 5.19]{BSTW}.
\end{proof}

We also record the following results, which are proved in \cite[Section 5.4]{BSTW}.

\begin{lemma} \label{lem:vanishing_H0}
Assume that condition \eqref{eq:irred} holds. Then:
\begin{enumerate}[label = \rm (\arabic*)]
    \item $\rmH^{0}(\QQ, T \hotimes C \hotimes \Lambda_{\cyc}^{\iota}) = 0$ for $C = \TT_{\bfh_{\frp}}^{\square}/T_{\frp}\TT_{\bfh_{\frp}}^{\square}, \tilTT_{\bfh_{\frp}}^{\star}/T_{\frp}\tilTT_{\bfh_{\frp}}^{\star}, \tilTT_{\bfh_{\frp}}^{\square}/\TT_{\bfh_{\frp}}^{\square}$ and $\tilTT_{\bfh_{\frp}}^{\star}/\tilTT_{\bfh_{\frp}}^{\square}$.
    \item $\rmH^{0}(\QQ_p, \Fil^{-}T \hotimes C \hotimes \Lambda_{\cyc}^{\iota}) = 0$ for $C = \tilTT_{\bfh_{\frp}}^{\star}/\TT_{\bfh_{\frp}}^{\square}$.
\end{enumerate}
\end{lemma}

\begin{remark}
    In this article, we follow the notation of \cite{CLW}, which differs from that of \cite{BSTW}. For the reader's convenience, the following table gives a dictionary between the two sets of notation:
    \begin{table}[H]
        \centering
        \begin{tabular}{c|c}
        In \cite{CLW} and our article     & In \cite{BSTW} \\ \hline
        $\frp$, $\Lambda_{\frp}$ and $\bfh_{\frp}$ & $v$, $\Lambda_{L}^{v}$ and $\bfh_{v}$ \\
        $\frh_{Mp^{\infty}}^{\ord}$ and $\frH_{Mp^{\infty}}^{\ord}$ & $\TT_{Mp^{\infty}}^{\ord}$ and $\HH_{Mp^{\infty}}^{\ord}$ \\
        $\rmH^{1}_{\ord}(Mp^{\infty})$ and $\widetilde{\rmH}^{1}_{\ord}(Mp^{\infty})$ & $\rmH^{1}_{\ord}(Mp^{\infty})$ and $\calH^{1}_{\ord}(Mp^{\infty})$ \\        
        $\TT_{\bfh_{\frp}}$ and $\TT_{\bfh_{\frp}}^{\square}$   & $\TT$ and $\TT_{1}$ \\
        $\tilTT_{\bfh_{\frp}}$ and $\tilTT_{\bfh_{\frp}}^{\square}$   & $\HH$ and $\HH_{1}$ \\
        $\TT_{\bfh_{\frp}}^{\star}$ & $\tilTT$ (and sometimes $\tilTT_{1}$) \\        
        $\tilTT_{\bfh_{\frp}}^{\star}$ & $\tilHH$ (and sometimes $\tilHH_{1}$) \\
        $\TT_{\bfh_{\frp}}^{\square}$ & $\II$ (by \cite[Theorem 5.19]{BSTW}).
        \end{tabular}
        \caption{Notations in our article compared with \cite{BSTW}}
        \label{tab:notations}
    \end{table}
\end{remark}

%% file: 03_beilinson_flach.tex
\section{Beilinson--Flach elements}
\label{sec:zeta}
With the preparations of the previous sections in place, we are now ready to define the zeta element for $f$ over $K$ and prove its reciprocity laws.

\subsection{Beilinson--Flach elements}
We follow \cite[Section~3.2]{CLW} closely, making the modifications needed for the $p$-ordinary case introduced in \cite[Section~5.2]{BSTW}.

Let $\alpha$ and $\beta$ be the roots of the Hecke polynomial $X^2-a_p(f)X+p^{k-1}$, and for $\xi\in\{\alpha,\beta\}$ let $f^\xi\in S_k(\Gamma_0(Np))$ be the $p$-stabilization of $f$ with $U_p$-eigenvalue $\xi$. We have assumed that $\alpha,\beta\in \calO$. Recall from Section \ref{sec:modforms} the $\calO$-lattice $T_f\subset V_f$. Working at level $Np$, we define $T_{f^\xi}\subset V_{f^\xi}$ as above, write $T_f^*={\rm Hom}_{\calO}(T_f,\calO)$, and define $T_{f^\xi}^*$ similarly.

We consider the following two Hida families and their specialisations.

(1) The form $f^{\alpha}$ is a specialization of a Hida family $\bff$ in the sense of \cite[Section 7]{KLZ}. Let $M(\bff)^{\ast}$ be the $G_{\QQ}$-module defined in \cite[Definition 7.2.5]{KLZ}. As explained in \cite[Section 7.3]{KLZ}, there is a specialisation map 
\[
\sp_{f^{\alpha}}: M(\bff)^{\ast} \otimes_{\Lambda_{\bff}} \calO \twoheadrightarrow T_{f^{\alpha}}^{\ast},
\] 
where $\Lambda_{\bff}$ is a localisation of Hida's ordinary Hecke algebra $\frh_{Np^{\infty}}^{\ord}$ at the maximal ideal corresponding to $\bff$, and the tensor product is with respect to the local $\calO$-homomorphism $\Lambda_{\bff} \rightarrow \calO$ that sends $T_{\ell}^{\prime}$ to $a_{\ell}(f)$ for all $\ell \neq p$ and maps $U_{p}^{\prime}$ to $\alpha$ (see \cite[Definition 2.4.3]{KLZ} and also \cite[Section 4.2.2]{BSTW} for the convention of Hecke operators).

(2) The canonical CM family $\bfh_{\frp}$ in Section \ref{sec:CM} corresponds to a branch (in the sense of \cite[Section 7.5]{KLZ}) of a Hida family $\bfh$ of tame level $D_{K}$. It follows from the definitions of $\widetilde{\TT}_{\bfh_{\frp}}$ and $M(\bfh)^{\ast}$ that there is a specialisation map
\[
\sp_{\bfh_{\frp}}: M(\bfh)^{\ast} \otimes_{\Lambda_{\bfh}} \Lambda_{\frp} \twoheadrightarrow \widetilde{\TT}_{\bfh_{\frp}},
\]
and the tensor product is with respect to the map induced by $\varphi: \frh_{D_{K}p^{\infty}} \rightarrow \Lambda_{\frp}$ from \cite[Section 4.2.4]{BSTW}.

Replacing $\ZZ_p(1)$ by the $p$-adic \'{e}tale sheaf ${\rm TSym}^{k-2}\scH_r(1)$ on $X_1(Mp^r)$ and $Y_1(Mp^r)$ introduced in \cite[\S{2.3}]{KLZ}, the inverse limits  \eqref{eq:big-H} define the $\frh^{\ord}_{Mp^\infty}[G_\QQ]$- and $\mathfrak{H}^{\ord}_{Mp^\infty}[G_\QQ]$-modules $H_{\ord}^1(Mp^{\infty})^{[k-2]}$ and $\widetilde{H}_{\ord}^1(Mp^\infty)^{[k-2]}$. It follows from \cite[Theorem 4.5.1, Remark 4.5.3]{KLZ} that these are identified with the twists of $H_{\ord}^1(Mp^{\infty})$ and $\widetilde{H}_{\ord}^1(Mp^\infty)$, respectively, by the $\ZZ_p$-linear twisting map 
\begin{equation}\label{eq:moment}
m_{k-2}:\Lambda_D\rightarrow\Lambda_D
\end{equation}
given by $z\mapsto z^{2-k}[z]$ for $z\in\ZZ_p^\times$. From these modules, we define $\TT_{\bfh_\frp}^{[k-2]}$ and $\widetilde{\TT}_{\bfh_\frp}^{[k-2]}$ by tensoring with $\varphi_v$ as in \eqref{def:CM-Tate-lattice}, and let $\TT_{\bfh_\frp}^{\square,[k-2]}$ and $\widetilde{\TT}_{\bfh_\frp}^{\square,[k-2]}$ denote the corresponding quotients by their torsion submodules.

We fix $M := N D_{K}$ and an integer $c>1$ with $(c,6pM)=1$, and let
\[
{}_{c} \cBF^{\bff, \bfh}_1 \in \rmH^{1}(\ZZ[1/p], M(\bff)^{\ast} \hotimes M(\bfh)^{\ast} \hotimes \tilLambda^{\iota})
\]
be the class associated in \cite[Section 8.1]{KLZ} with the Hida families $\bff$ and $\bfh$. Let 
\begin{equation}\label{eq:BF-unb}
_c \cBF^{\alpha,\bfh_{\frp}}\in \rmH^1(\ZZ[1/p],T_{f^\alpha}^*\hotimes\bigl(\widetilde{\TT}_{\bfh_\frp}^{\square,[k-2]}\bigr)\hotimes \tilLambda^{\iota})
\end{equation}
be the image of ${}_{c} \cBF^{\bff, \bfh}_1$ under the specialization maps $\sp_{f^{\alpha}}$ and $\sp_{\bfh_{\frp}}$.

By \cite[Proposition 7.3.1]{KLZ}, there is an isomorphism $({\rm Pr}^\alpha)_*:T_{f^\alpha}^*\xrightarrow{\sim}T_f^*$ induced by a $p$-stabilisation map ${\rm Pr}^\alpha$; we shall therefore view the Beilinson--Flach elements \eqref{eq:BF-unb} 
as classes in
\[
_c \cBF^{\alpha,\bfh_{\frp}} \in \rmH^1(\ZZ[1/p],T_{f}^*\hotimes(\widetilde{\TT}_{\bfh_\frp}^{\square,[k-2]})\hotimes \tilLambda^\iota).
\] 
Using the further identification $T_f^{\ast} \cong T$ (since $f$ has trivial nebentypus, so its complex conjugate $\bar{f}$ is $f$ itself), we have
\[
_c \cBF^{\alpha,\bfh_{\frp}} \in \rmH^1(\ZZ[1/p], T \hotimes(\widetilde{\TT}_{\bfh_\frp}^{\square,[k-2]})\hotimes \tilLambda^\iota)
\]

Let $\scR := \calO \hotimes_{\ZZ_p} \Lambda_{\frp} \hotimes \Lambda_{\cyc}$. As observed in \cite[\S{5.3.3}]{BSTW}, we may find $c$ as above such that 
\[
\mathbf{c}:=(c^2-\epsilon_K(c)\otimes\langle c\rangle\gamma_v^{-r_c})\otimes\gamma_{\rm cyc}^{2r_c} \in \scR
\] 
is invertible, where $r_c=\log_p(c)$. After choosing such a $c$, we set
\[
\cBF^{\alpha,\bfh_{\frp}}:=\mathbf{c}^{-1}\cdot{}_c\cBF^{\alpha,\bfh_{\frp},\eta}\in\rmH^1(\ZZ[1/p],T\hotimes(\widetilde{\TT}_{\bfh_\frp}^{\square, [k-2]})\hotimes \tilLambda^\iota),
\]
which is independent of $c$.

After twisting by the inverse of the character $m_{k-2}$ of \eqref{eq:moment}, we shall henceforth view $_c \cBF^{\alpha,\bfh_{\frp}} \in \rmH^1(\ZZ[1/p],T \hotimes(\widetilde{\TT}_{\bfh_\frp}^{\square})\hotimes \tilLambda^\iota)$. For a tame character $\eta: \Delta \rightarrow \calO^{\times}$, we write its projection along $\eta$ as
\begin{equation}\label{eq:BF-eta}
\cBF^{\alpha,\bfh_{\frp}, \eta}\in \rmH^1(\ZZ[1/p],T \hotimes\bigl(\widetilde{\TT}_{\bfh_\frp}^{\square}\bigr)\hotimes \Lambda_{\cyc}^{\iota}).
\end{equation}

The following proposition records the local properties of $\cBF^{\alpha,\bfh_{\frp}}$. It follows from \cite[Proposition 8.17]{KLZ} which applies to modular forms in general and is recorded in \cite[Lemmas 5.1 and 5.2, Corollary 5.3]{BSTW} for the special case of modular forms attached to elliptic curves.

\begin{proposition} \label{prop:local_BF}
Let $\loc_{p}: \rmH^1(\ZZ[1/p],T\hotimes(\widetilde{\TT}_{\bfh_\frp}^{\square})\hotimes \Lambda_{\cyc}^\iota) \rightarrow \rmH^1(\QQ_p,T\hotimes(\widetilde{\TT}_{\bfh_\frp}^{\square})\hotimes \Lambda_{\cyc}^\iota)$ be the localization map at $p$. Then the following statements hold.
    \begin{enumerate}[label = \rm (\arabic*)]
        \item The image of $\loc_{p}({}_{c}\cBF^{\alpha,\bfh_{\frp},\eta})$ in  $\rmH^1(\QQ_p, \Fil^{-} T \hotimes(\Fil^{-} \widetilde{\TT}_{\bfh_\frp}) \hotimes \Lambda_{\cyc}^\iota)$ is zero.
        \item The image of $\loc_{p}({}_{c}\cBF^{\alpha,\bfh_{\frp},\eta})$ in  $\rmH^1(\QQ_p,T\hotimes(\Fil^{-} \widetilde{\TT}_{\bfh_\frp}) \hotimes \Lambda_{\cyc}^\iota)$ is contained in 
        \[ \rmH^1(\QQ_p, \Fil^{+}T \hotimes(\Fil^{-} \widetilde{\TT}_{\bfh_\frp}) \hotimes \Lambda_{\cyc}^\iota).
        \]
        In particular, ${}_{c}\cBF^{\alpha,\bfh_{\frp},\eta} \in \rmH^{1}_{\rel, \ord}(\ZZ[1/p], T \hotimes \widetilde{\TT}_{\bfh_\frp}^{\square}) \hotimes \Lambda_{\cyc}^\iota)$
        \item The image of $\loc_{p}({}_{c}\cBF^{\alpha,\bfh_{\frp},\eta})$ in  $\rmH^1(\QQ_p,\Fil^{-} T \hotimes(\widetilde{\TT}_{\bfh_\frp}^{\square}) \hotimes \Lambda_{\cyc}^\iota)$ is contained in 
        \[ \rmH^1(\QQ_p, \Fil^{-}T \hotimes(\Fil^{+} \widetilde{\TT}_{\bfh_\frp}) \hotimes \Lambda_{\cyc}^\iota).
        \]
    \end{enumerate}
\end{proposition}

\subsection{Explicit reciprocity laws (I)} \label{sec:ERL-I}
Since the explicit reciprocity laws of \cite{KLZ} hold for the Hida families $\bff$ and $\bfh$ before specialization at $f^{\alpha}$, the same arguments carry over; see \cite[Section~5.3]{BSTW}. We now record the corresponding higher-weight results and provide proofs for the reader's convenience.

We consider two sets $\Xi^{\one}$ and $\Xi^{\two}$ of continuous characters $\xi: \Gamma_{\frp} \hotimes \Gamma_{\cyc} \rightarrow \barQQ_p^{\times}$. Each such character determines a continuous $\calO$-homomorphism $\phi_\chi: \scR \rightarrow \barQQ_p$, where $\calO$ is identified with the completion of $\iota_p(\calO)$ inside $\barQQ_p$. Note that any character $\chi: \Gamma_{\frp} \hotimes \Gamma \rightarrow \barQQ_p^{\times}$ is determined by the pair $(\chi_1, \chi_2) := (\chi|_{\Gamma_{\frp}}, \chi|_{\Gamma_{\cyc}})$. We define
 \begin{equation}\label{eq:ch-1}
 \Xi^{\one}:= \{ \chi = (1,\chi_2) \ : \ \text{$\chi_2 = \lrangle{\epsilon}^{j} \psi_\zeta$ for some root of unity $\zeta$ of exact order $p^{t}$ and $0 \leq j \leq k-2$}\},
 \end{equation} 
where $\psi_{\zeta}: G_{\QQ} \twoheadrightarrow \Gamma \rightarrow \barQQ_p^{\times}$ is the character determined by $\psi_{\zeta}(\gamma) = \zeta$, and $\epsilon = \omega \lrangle{\epsilon}$ is the Teichmüller decomposition of the $p$-adic cyclotomic character $\epsilon$.

\subsubsection{The Coleman map}
Following the construction in \cite[Sections 8, 10]{KLZ} and particularly \cite[Section 5.3.1]{BSTW}, there is an injective $\scR$-morphism (denoted by $\scC$ in \cite[Equation (5.5)]{BSTW})
\[
\Cole_{\frp}: \rmH^{1}(\QQ_p, \Fil^{-}T \hotimes (\Fil^{+} \widetilde{\TT}_{\bfh_\frp}) \hotimes \Lambda_{\cyc}^\iota) \rightarrow J_{f} \otimes_{\calO} \scR \subseteq \scR[1/\varpi],
\]
where, under our running hypothesis \eqref{eq:irred}, $J_{f}$ is the fractional ideal of $\calO$ given by
\[
J_{f} = \dfrac{1}{\alpha(1-p^{k-2}\alpha^{-2})(1-p^{k-1}\alpha^{-2}) \lambda_{N}(f) c_f} \calO
\]
where $c_{f}$ is the congruence number of $f$ as in \cite[Section 2.2.3]{BSTW}, and $\lambda_{N}(f) = \pm 1$ is the eigenvalue of $f$ for the usual Atkin--Lehner involution $w_N$ of level $N$. More explicitly, using the notation in \cite[Section 10]{KLZ} (see the proof of \cite[Theorem 5.5]{BSTW}), the Coleman map $\Cole_{\frp}$ can be described as
\begin{equation} \label{eq:coleman_definition}
\Cole_{\frp} \left( \loc_{p} \cBF^{\alpha, \bfh_{\frp}, \eta} \right) = \sp \left(\lrangle{\calL\left( \cBF^{\bff, \bfh_{\frp}} \right), \eta_{\bff} \otimes \omega_{\bfh_{\frp}}} \right)
\end{equation}
where $\sp$ is the specialization map
\[
\sp: I_{\bff} \hotimes \Lambda_{\bfh_{\frp}}^{\cusp} \twoheadrightarrow J_{f} \hotimes \Lambda_{\frp} \hotimes \Lambda_{\cyc} = J_{f} \otimes_{\calO} \scR
\]
induced by the specialization map $\Lambda_{\bff} \rightarrow \calO$ corresponding to $f_{\alpha}$ and the projection $\Lambda_{\bfh_{\frp}}^{\cusp} \twoheadrightarrow \Lambda_{\bfh_{\frp}} = \Lambda_{\frp}$.

We renormalize $\Cole_{\frp}$ by
\[
\Cole_{\frp}^{\inte} := c_{f} \lambda_{N}(f) \cdot (1-p^{k-2}\alpha^{-2})(1-p^{k-1}\alpha^{-2}) \cdot \Cole_{\frp}: \rmH^{1}(\QQ_p, \Fil^{-}T \hotimes (\Fil^{+} \widetilde{\TT}_{\bfh_\frp}) \hotimes \Lambda_{\cyc}^\iota) \rightarrow \scR.
\]
We define the \emph{first $p$-adic $L$-function of $f$ over $K$} as
\[
\calL_{p}^{\one}(\cBF^{\alpha, \bfh_{\frp},\eta}) := \Cole_{\frp}^{\inte}(\loc_{p}(\cBF^{\alpha, \bfh_{\frp},\eta})) \in \scR.
\]
This is well-defined by Proposition \ref{prop:local_BF}(3).

\subsubsection{Explicit reciprocity law (I)}
This $p$-adic $L$-function is related to $L$-values by the following explicit reciprocity law, which follows from \cite[Theorems 2.7.4, 7.7.2, 10.2.2]{KLZ}. It is stated in \cite[Theorem 5.10]{BSTW}. The same proof gives the following theorem (see also \cite[Theorem 3.6 and Remark 4.6]{CLW}).

\begin{theorem}[Explicit Reciprocity Law I]\label{ERLI-thm}
The element $\calL_{p}^{\one}(\cBF^{\alpha, \bfh_{\frp}}) \in \scR$ has the following interpolation property: for $\chi\in \Xi^{\one}$, 
\begin{multline*}
\phi_\chi(\calL_{p}^{\one}(\cBF^{\alpha, \bfh_{\frp},\eta})) = {\calE_{p,\alpha,\eta}(f,\chi)^2} \dfrac{(-1)^{j+1} (j!)^2}{2^{2j-k+1} (-i)^{k-1}} \\ \times \dfrac{L(f \otimes \omega^{j}\eta \psi_{\zeta}^{-1}, j+1) L((f \otimes \epsilon_{K}) \otimes \omega^{j}\eta \psi_{\zeta}^{-1}, j+1)}{\frg(\omega^{j}\eta\psi_{\zeta}^{-1})^{2} \pi^{2j+2-k} \Omega_{f}^{\can}}
\end{multline*}
where $\Omega_{f}^{\can} := c_{f}^{-1} (4\pi)^{k}\lrangle{f,f}$ is Hida's canonical period (see \cite[Section 6]{CastellaHsieh}), and $\calE_{p,\alpha,\eta}(f,\chi)$ is the Euler factor defined in \eqref{eq:Euler-I}.
\end{theorem}

\begin{proof}
 This follows essentially from \cite[Theorems 2.7.4, 7.7.2, 10.2.2]{KLZ}. Let $\calL_{p}^{\HP}(\bff, \bfh_{\frp}, 1+\bfj)$ be the three-variable $p$-adic Rankin $L$-function in \cite[Theorems 2.7.4, 7.7.2]{KLZ} of Hida and Panchishkin. On the other hand, let
 \begin{equation} \label{eq:motivic_1_defn}
 \calL_{p}^{\mot}(\bff, \bfh_{\frp}, 1+\bfj) := (-1)^{1+\bfj} \lambda_{N}(\bff) \lrangle{\calL\left( \cBF^{\bff, \bfh_{\frp}} \right), \eta_{\bff} \otimes \omega_{\bfh_{\frp}}}
 \end{equation}
 be the three-variable motivic $p$-adic $L$-function defined as in \cite[Theorem 10.2.2]{KLZ}. It then follows from \cite[Theorem 10.2.2]{KLZ} that
 \[
 \calL_{p}^{\mot}(\bff, \bfh_{\frp}, 1+\bfj) = \calL_{p}^{\HP}(\bff, \bfh_{\frp}, 1+\bfj).
 \]
 Here $\calL_{p}^{\HP}(\bff, \bfh_{\frp}, 1+\bfj)$ is the Rankin--Selberg $p$-adic $L$-function defined by Hida and Panchishkin, with the interpolation property recorded in \cite[Theorem 2.7.4]{KLZ}. Let $\calL_{p}^{\mot}(f^{\alpha}, \bfh_{\frp}, 1+\bfj)$ denote the image of $\calL_{p}^{\mot}(\bff, \bfh_{\frp}, 1+\bfj)$ under the specialization map at $f^{\alpha}$. Then, for all $\eta: \Delta \rightarrow \calO^{\times}$ and $\chi \in \Xi^{\one}$ as in the statement, we have 
 \begin{multline}   \label{eq:mot_interpolation_1}  
\phi_{\eta\chi}\left( \calL_{p}^{\mot}(f^{\alpha}, \bfh_{\frp}, 1+\bfj) \right) = \dfrac{\calE_{p,\alpha,\eta}(f,\chi)^2}{(1-\alpha^{-2}p^{k-2})(1-\alpha^{-2}p^{k-1})} \dfrac{(j!)^{2}}{2^{2j+k+1} (-i)^{k-1}} \\ \times \dfrac{L(f \otimes \omega^{j}\eta \psi_{\zeta}^{-1}, \bfh_{\frp, \chi_1}^{\circ}, j+1)}{\frg(\omega^{j}\eta\psi_{\zeta}^{-1})^{2} \pi^{2j+2} \lrangle{f,f}_{N}}
 \end{multline}
where
\begin{equation} \label{eq:Euler-I}
    \calE_{p,\xi,\eta}(f,\chi) := \dfrac{p^{t^{\prime}(j+1)}}{\xi^{t^{\prime}}} \cdot \left(1 - \dfrac{\psi_{\zeta}^{-1}\eta\omega^{j}(p) p^{k-2-j}}{\xi} \right) \left(1 - \dfrac{\psi_{\zeta}\eta^{-1}\omega^{-j}(p) p^{j}}{\xi} \right)
\end{equation}
with $t^{\prime} = 0$ or $t+1$ according to whether or not $\eta = \omega^{-j}$ and $t=0$.

Combining \eqref{eq:coleman_definition} with \eqref{eq:motivic_1_defn}, the interpolation formula \eqref{eq:mot_interpolation_1} gives
\begin{align*} 
\phi_{\chi} \left( \calL_{p}^{\one}(\cBF^{\alpha, \bfh_{\frp}, \eta}) \right) &= c_{f} \cdot {\calE_{p,\alpha,\eta}(f,\chi)^2} \dfrac{(-1)^{j+1}(j!)^{2}}{2^{2j+k+1} (-i)^{k-1}} \dfrac{L(f \otimes \omega^{j}\eta \psi_{\zeta}^{-1}, \bfh_{\frp, \chi_1}^{\circ}, j+1)}{\frg(\omega^{j}\eta\psi_{\zeta}^{-1})^{2} \pi^{2j+2} \lrangle{f,f}_{N}} \\
&= {\calE_{p,\alpha,\eta}(f,\chi)^2} \dfrac{(-1)^{j+1}(j!)^{2}}{2^{2j-k+1} (-i)^{k-1}} \dfrac{L(f \otimes \omega^{j}\eta \psi_{\zeta}^{-1}, \bfh_{\frp, \chi_1}^{\circ}, j+1)}{\frg(\omega^{j}\eta\psi_{\zeta}^{-1})^{2} \pi^{2j+2-k} \Omega_{f}^{\can}}
\end{align*}
Finally, as in \cite[Remark 4.6]{CLW}, we note that the specialization of $\bfh_{\frp}$ at the trivial character is the ordinary $p$-stabilization of the weight one Eisenstein series $\bfh_{\frp, \chi_1}^{\circ} = E_1(1, \epsilon_{K})$. Thus it follows that
\[
L(f \otimes \omega^{j}\eta \psi_{\zeta}^{-1}, \bfh_{\frp, \chi_1}^{\circ}, j+1) = L(f \otimes \omega^{j}\eta \psi_{\zeta}^{-1}, j+1) L((f \otimes \epsilon_{K}) \otimes \omega^{j}\eta \psi_{\zeta}^{-1}, j+1)
\]
The theorem is thus proved.
\end{proof}

\subsubsection{Cyclotomic descents}
We first recall the cyclotomic $p$-adic $L$-function of $f$ associated with the cyclotomic $\ZZ_p$-extension $\QQ_{\infty}/\QQ$. We use the $p$-adic $L$-function defined by the explicit reciprocity law for the Beilinson--Kato zeta element in \cite[Theorem 16.6]{Kato}, reformulated in \cite[Theorem 2.15]{CLW} (which is denoted by $L_{\alpha}(f)^{\eta}$ in \textit{op.cit.}).

\begin{theorem}[Cyclotomic $p$-adic $L$-function] \label{thm:cyclotomic_p-adic_L-function}
    There is an element $\calL_p(f/\QQ)^{\eta} \in \Lambda$ satisfying the following interpolation property: for all characters $\chi = \lrangle{\epsilon}^{j} \psi_{\zeta}$ of $\Gamma$ with $0 \leq j \leq k-2$ and $\zeta$ a primitive $p^{t}$-th root of unity for some $t \geq 0$, we have
    \[
    \phi_{\chi}(\calL_{p}(f/\QQ)^{\eta}) = \calE_{p,\alpha,\eta}(f,\chi) \dfrac{L(f \otimes \psi_{\zeta}^{-1}\eta\omega^{j}, j+1)}{(-2\pi i)^{j+1} \frg(\psi_{\zeta}^{-1}\eta\omega^{j}) \Omega_{f}^{\pm}}. 
    \]
where $\Omega_{f}^{\pm}$ is the $p$-normalized period in \cite[Definition 2.11]{CLW}, with the sign determined by $\pm = (-1)^{j}$.
\end{theorem}

We \emph{define} \emph{Perrin-Riou's $p$-adic $L$-function} of $f$ over $K$ as
\[
\calL_p^{\PR}(f/K)^{\eta} := \calL_{p}^{\one}(\cBF^{\alpha, \bfh_{\frp},\eta}) \in \scR.
\]
Its projection to the cyclotomic line recovers the cyclotomic $p$-adic $L$-functions of $f$ and of the quadratic twist $f^{K}$ over $\QQ$ (under the identification $\Gamma_{\cyc} \simeq \calO\lrbracket{\Gal(\QQ_{\cyc}/\QQ)}$).

\begin{proposition} \label{prop:cyclotomic_descent}
Let $I_{\ac}$ be the kernel of the natural projection $\phi_{\cyc}: \scR \twoheadrightarrow \Lambda_{\cyc}$. Suppose \eqref{eq:irred} holds. Then we have
    \[
    (\calL_p^{\PR}(f/K)^{\eta} \bmod{I_{\ac}}) = (\calL_p(f/\QQ)^{\eta} \calL_p(f^{K}/\QQ)^{\eta}).
    \]
\end{proposition}

\begin{proof}
    This is proved in \cite[Equation (4.9)]{CLW} \footnote{This result also appears as \cite[Proposition 5.25]{BSTW}.}: compare the interpolation formula in Theorem \ref{ERLI-thm} and Theorem \ref{thm:cyclotomic_p-adic_L-function}, and note that
    \[
    \Omega_{f}^{\can} \sim_{p} \Omega_{f}^{+} \cdot \Omega_{f}^{-}, \quad \Omega_{f}^{\pm} \sim_{p} \Omega_{f^{K}}^{\mp}
    \]
    following \cite[Section 4.4]{DDT} and \cite[Lemma 9.6]{SZ14}.
\end{proof}

With this comparison theorem, we have the following nonvanishing result for the Perrin-Riou $p$-adic $L$-function $\calL_{p}^{\PR}(f/K)^{\eta}$.
\begin{corollary} \label{coro:cyclo_nonvanishing}
    We have $(\calL_{p}^{\PR}(f/K)^{\eta}) \bmod{I_{\ac}}) \neq 0$.
\end{corollary}
\begin{proof}
    It follows from Proposition \ref{prop:cyclotomic_descent} that it suffices to show the nonvanishing of the cyclotomic $p$-adic $L$-functions $\calL(f/\QQ)^{\eta}$ and $\calL_p(f^{K}/\QQ)^{\eta}$. The latter is a classical result of Rohrlich \cite{rohrlich}  and Shimura \cite[Proposition 2]{shimura} (for $k > 2$).
\end{proof}

\subsection{Explicit Reciprocity law (II)}
The second explicit reciprocity law essentially arises by exchanging the roles of $f$ and $\bfh_{\frp}$ in the preceding analysis. Following \cite[Section 5.3.2]{BSTW}, there is an injective $\scR$-homomorphism (denoted by $\scL$ in \cite[page 47]{BSTW})
\[
\Log_{\bar{\frp}}: \rmH^{1}(\QQ_p, \Fil^{+}T \hotimes (\Fil^{-} \widetilde{\TT}_{\bfh_\frp}^{\star}) \hotimes \Lambda_{\cyc}^\iota) \rightarrow \tilI_{\bfh_{\frp}} \otimes_{\Lambda_{\frp}} \scR,
\]
where $I_{\bfh_{\frp}}$ is the congruence ideal for the canonical CM family $\bfh_{\frp}$, introduced in \cite[Section 4.3.1]{BSTW},  $I_{\bfh_{\frp}}^{\star}$ is its reflexive closure, and $\Fil^{-} \widetilde{\TT}_{\bfh_\frp}^{\star} := \tilTT_{\bfh_{\frp}}^{\star}/ \widetilde{\TT}_{\bfh_\frp, \frp}$. To describe $I_{\bfh_{\frp}}^{\star}$, we introduce Katz's $p$-adic $L$-function, following \cite[Theorem 2.2.1]{Cas25}, with the functional equation following from \cite[Theorem II.6.4]{dShalit}. \footnote{There seems to be a typographical error before \cite[Equation (4.1)]{CLW} on the range of $a$ and $b$.}

\begin{theorem}
    There is an element $\calL_{\barfrp}^{\Katz}(K) \in \Lambda_{K}^{\ur}$ characterised by the following interpolation property: for every character $\xi$ of $\Gamma_{K}$ crystalline at both $\barv$ and $v$ corresponding to a Hecke character of infinite type $(b,a)$ with $a > 0$ and $b \leq 0$, it satisfies
\begin{equation} \label{eq:Katz_interpolation}
\calL_{\barfrp}^{\Katz}(K)(\xi) = \left( \dfrac{\Omega_{p}}{\Omega_{\infty}}\right)^{a-b} (a-1)! \left( \dfrac{\sqrt{D_K}}{2\pi} \right)^{b} (1-\xi^{-1}(\barfrp)p^{-1})(1-\xi(\frp)) L(\xi, 0),
\end{equation}
Moreover, we have the functional equation 
\begin{equation} \label{eq:Katz_functional_equation}
\calL_{p}^{\Katz}(K)(\xi) = \calL_p^{\Katz}(\xi^{-1}\bfN^{-1}),
\end{equation}
where the equality is up to $p$-adic unit.
\end{theorem}
Using Katz's $p$-adic $L$-function $\calL_{\barfrp}^{\Katz}(K)$, we define 
\begin{itemize}
    \item Katz's anticyclotomic $p$-adic $L$-function $\calL_{\barfrp}^{\Katz, \ac}(K)$ to be its image in $\Lambda_{K}^{\ur}$ under the homomorphism $\pr_{\ac}: \Gamma_{K} \rightarrow \Gamma_{\ac}$ given by $g \mapsto g^{1-c}$;
    \item its further image $\calL_{\barfrp}^{\Katz, \frp}(K)$ under the identification $\iota_{\ac}: \Lambda_{\ac}^{\ur} \rightarrow \Lambda_{\frp}^{\ur}$ defined in Section \ref{sec:more-iwasawa-algebra}. This is Katz's $p$-adic $L$-function recalled in \cite[Theorem 4.19]{BSTW}.
\end{itemize}

Returning to the second reciprocity law, by \cite[Theorem 4.22]{BSTW}, the reflexive closure is
\[
I_{\bfh_{\frp}}^{\star} = \dfrac{1}{H_{\frp}} \Lambda_{\frp}, \text{ with } (H_{\frp}) = \left(\dfrac{h_{K}}{w_{K}} \cdot T_{\frp} \cdot \calL^{\Katz}_{\barfrp}(K)\right) \text{ in } \Lambda_{\frp}^{\ur}, 
\]
where $w_{K}$ is the number of roots of unity in $\calO_{K}^{\times}$ and $\calL_{\barfrp}^{\Katz, \frp}(K)$ is Katz's $p$-adic $L$-function defined above. As in \cite[Section 5.3.3]{BSTW}, we renormalize
\[
\Log_{\bar{\frp}}^{\inte} := H_{\frp} \cdot \Log_{\bar{\frp}}: \rmH^{1}(\QQ_p, \Fil^{+}T \hotimes (\Fil^{-} \widetilde{\TT}_{\bfh_\frp}^{\star}) \hotimes \Lambda_{\cyc}^\iota) \otimes_{\Lambda_{\frp}} \Lambda_{\frp}^{\ur} \rightarrow \scR^{\ur},
\]

More explicitly, using the notation in \cite[Section 10]{KLZ} (see the proof of \cite[Theorem 5.7]{BSTW}), the logarithm map $\Log_{\frp}$ can be described as
\begin{equation} \label{eq:log_definition}
\Log_{\frp} \left( \loc_{p} \cBF^{\alpha, \bfh_{\frp}, \eta} \right) = \sp^{\prime} \left( \lrangle{\calL\left( \cBF^{\bff, \bfh_{\frp}} \right), \eta_{\bfh_{\frp}}\otimes \omega_{\bff} } \right)
\end{equation}
where $\sp^{\prime}$ is the specialization map
\[
\sp^{\prime}: I_{\bfh_{\frp}} \hotimes \Lambda_{\bff}^{\cusp} \hotimes \Lambda_{\cyc} \twoheadrightarrow I_{\bfh_{\frp}}^{\star} \hotimes \Lambda_{\cyc} = I_{\bfh_{\frp}}^{\star} \otimes_{\Lambda_{\frp}} \scR
\]
induced by the specialization map $\Lambda_{\bff}^{\cusp} \rightarrow \calO$ corresponding to $f_{\alpha}$.

We define the preparatory \emph{second $p$-adic $L$-function of $f$ over $K$} as
\[
\widetilde{\calL_{p}^{\two}}(\cBF^{\alpha, \bfh_{\frp}, \eta}) := \Log_{\barfrp}^{\inte}(\loc_{p}(\cBF^{\alpha, \bfh_{\frp}, \eta})) \in \scR^{\ur}.
\]
This is well-defined by Proposition \ref{prop:local_BF}(2).

We define \footnote{It seems to be a typographical error in \cite[Section 5.3]{BSTW}, which is also inherited in \cite[Section 3.5.6]{CLW} when defining $\Xi^{\two}$, that $\chi_{1}(\gamma_{\frp}^{h_p})$ should be written as $\chi_1(\gamma_{\frp}^{p^{h_p}})$ instead.}
\begin{equation}\label{eq:ch-2}
\Xi^{\two}:=
\left\{
\chi=(\chi_1,\chi_2)\;:\;
\begin{aligned}
&\chi_1(\gamma_{\frp}^{p^{h_p}}) = \lrangle{\epsilon(\gamma_{\cyc})}^m
 &&\text{for some $m \geq k$ and $m \equiv 0 \pmod{p-1}$},\\
&\chi_2 = \zeta \lrangle{\epsilon}^n
 &&\text{for $\zeta \in \mu_{p^{\infty}}$ of exact order $p^{t}$ and $k-1 \leq n \leq m-1$}.
\end{aligned}
\right\}.
\end{equation}
Given $\chi\in \Xi^{\two}$, let 
$\psi_\chi$ be the algebraic Hecke character of $K$ with infinity type $(n-m,n)$ and such that
$\sigma_{\psi_\chi}$ is the composition 
\[
G_K \twoheadrightarrow \Gamma_{\frp}\times\Gamma
\xrightarrow{\chi_1\times \chi_2^{-1}} \barQQ_p^\times.
\]
Then the second explicit reciprocity law can be rewritten as follows, with the proof proceeding as in \cite[Theorems 5.7, 5.11]{BSTW}

\begin{theorem}[Explicit Reciprocity Law II]\label{thm:ERLIIint}
The element $\widetilde{\calL}_p^{\two}(\cBF^{\alpha, \bfh_{\frp}})\in \scR^\ur$ has the following interpolation property: for $\chi\in \Xi^{\two}$, 
\begin{multline} \label{eq:ERLII}
    \phi_{\chi} \left( \widetilde{\calL_{p}^{\two}}(\cBF^{\alpha, \bfh_{\frp}, \eta}) \right) = \dfrac{\calE_{p, \eta}(\bfh_{\frp}, f,\chi)}{(1-p^{m-1} \chi_1(\barfrp)^{-2})(1-p^{m}\chi_1(\barfrp)^{-2})} \dfrac{(-1)^{n+1}(n!)(n+1-k)!}{\pi^{2n+3-k} (-i)^{m+1-k} 2^{2n+m+3-k}} \\ \times \dfrac{h_{K}}{w_K} \cdot \phi_{\chi}(T_{\frp} \cdot \calL_{\barfrp}^{\Katz}(K)) \cdot \dfrac{L(\bfh_{\frp, \chi_1}^{\circ}, f \otimes \omega^{n}\eta \psi_{\zeta}^{-1}, n+1)}{\lrangle{\bfh_{\frp, \chi_1}^{\circ},\bfh_{\frp, \chi_1}^{\circ}}_{D_{K}}},
\end{multline}
where $\calE_{p,\eta}(\bfh_{\frp}, f, \chi)$ is the Euler factor defined in \eqref{eq:Euler-II}.
\end{theorem}

\begin{proof}
We use the same notation as in the proof of Theorem \ref{ERLI-thm}, but reverse the roles of $\bff$ and $\bfh_{\frp}$. We put
 \begin{equation} \label{eq:mot_definition_2}
 \calL_{p}^{\mot}(\bfh_{\frp}, \bff, 1+\bfj) := (-1)^{1+\bfj} \lambda_{N}(\bfh_{\frp}) \lrangle{\calL\left( \cBF^{\bfh_{\frp}, \bff} \right), \eta_{\bfh_{\frp}} \otimes \omega_{\bff}}
 \end{equation}
 be the three-variable motivic $p$-adic $L$-function defined as in \cite[Theorem 10.2.2]{KLZ}. It then follows from \cite[Theorem 10.2.2]{KLZ} that
 \[
 \calL_{p}^{\mot}(\bfh_{\frp}, \bff, 1+\bfj) = \calL_{p}^{\HP}(\bfh_{\frp}, \bff, 1+\bfj).
 \]
 By the interpolation property of $\calL_{p}^{\HP}(\bfh_{\frp}, \bff,  1+\bfj)$ given in \cite[Theorem 2.7.4]{KLZ}, if $\calL_{p}^{\mot}(\bfh_{\frp}, f^{\alpha},  1+\bfj)$ denotes the image of $\calL_{p}^{\mot}(\bfh_{\frp}, \bff, 1+\bfj)$ under the specialization map at $f^{\alpha}$, then for all $\eta: \Delta \rightarrow \calO^{\times}$ and $\chi \in \Xi^{\two}$ as in the statement, we have \footnote{More precisely, with the notation of \cite[Theorem 2.7.4]{KLZ}, the weight of $\bfh_{\frp, \chi_1}$ is $r := m+1$, while the weight of $f$ is $r^{\prime} = k$. Hence $s = n+1$, and the criticality range becomes $k-1 \leq n \leq m-1$ in \cite[Theorem 2.7.4]{KLZ}. This explains why we impose a smaller range in \eqref{eq:ch-2} than in \cite{BSTW, CLW}.}
\begin{multline} \label{eq:mot_interpolation_2}
\phi_{\eta\chi}\left( \calL_{p}^{\mot}(\bfh_{\frp}, f^{\alpha}, 1+\bfj) \right) = \dfrac{\calE_{p, \eta}(\bfh_{\frp}, f,\chi)}{(1-p^{m-1} \chi_1(\barfrp)^{-2})(1-p^{m}\chi_1(\barfrp)^{-2})} \dfrac{(n!)(n+1-k)!}{\pi^{2n+3-k} (-i)^{m+1-k} 2^{2n+m+3-k}} \\ \times \dfrac{L(\bfh_{\frp, \chi_1}^{\circ}, f \otimes \omega^{n}\eta \psi_{\zeta}^{-1}, n+1)}{\lrangle{\bfh_{\frp, \chi_1}^{\circ},\bfh_{\frp, \chi_1}^{\circ}}_{D_{K}}}
\end{multline}
where
\begin{equation}\label{eq:Euler-II}
\mathcal{E}_{p,\eta}(\bfh_{\frp},f,\chi)=\begin{cases}\Bigl(1-\frac{p^n}{\chi_1(\barfrp)\alpha}\Bigr)\Bigl(1-\frac{p^n}{\chi_1(\barfrp)\beta}\Bigr)\Bigl(1-\frac{p^{m}\alpha}{p^{n+1}\chi_1(\barfrp)}\Bigr)\Bigl(1-\frac{p^{m}\beta}{p^{n+1}\chi_1(\barfrp)}\Bigr), \, &\text{if $\eta=\omega^{-n}$ and $t=0$}
\\[0.6em]
\Bigl(\frac{p^{t+1}}{\mathfrak{g}(\eta\omega^n\psi_\zeta^{-1})}\Bigr)^2\Bigl(\frac{p^{2n+1-k}}{\chi_1(\barfrp)^2}\Bigr)^{t+1}, \, &\text{otherwise}.
\end{cases}
\end{equation}

Combining \eqref{eq:log_definition} with \eqref{eq:mot_definition_2}, the interpolation formula \eqref{eq:mot_interpolation_2} gives
\begin{multline*}
\phi_{\chi} \left( \widetilde{\calL_{p}^{\two}}(\cBF^{\alpha, \bfh_{\frp}, \eta}) \right) = \dfrac{\calE_{p, \eta}(\bfh_{\frp}, f,\chi)}{(1-p^{m-1} \chi_1(\barfrp)^{-2})(1-p^{m}\chi_1(\barfrp)^{-2})} \dfrac{(-1)^{n+1} (n!)(n+1-k)!}{\pi^{2n+3-k} (-i)^{m+1-k} 2^{2n+m+3-k}} \\ \times \dfrac{h_{K}}{w_K} \cdot \phi_{\chi}(T_{\frp} \cdot \calL_{\barfrp}^{\Katz}(K)) \cdot \dfrac{L(\bfh_{\frp, \chi_1}^{\circ}, f \otimes \omega^{n}\eta \psi_{\zeta}^{-1}, n+1)}{\lrangle{\bfh_{\frp, \chi_1}^{\circ},\bfh_{\frp, \chi_1}^{\circ}}_{D_{K}}},
\end{multline*}
as desired.
\end{proof}

\subsubsection{Anticyclotomic descent}
We recall the generalized Heegner hypothesis used in \cite{CLW}. We write $N=N^+N^-$ with $N^{+}$ (resp. $N^{-}$) divisible only by primes that are split or ramified (resp. inert) in $K$. We assume the \emph{general Heegner hypothesis}:
\begin{equation}\label{eq:gen-Heegner} \tag{gen-Heeg}
 N^-=1, \quad \epsilon_\ell(\BC_K(f))=+1 \text{ for every nonsplit }\ell\mid N^+,
\end{equation}
where $\epsilon_\ell(\BC_K(f))$ is the local root number at $v$ normalised as in \cite[page 710]{Hsieh2014}. Recall that the \emph{classical Heegner hypothesis} is that every prime dividing $N$ splits in $K$. In this case, $N^{-} = 1$ and there is no nonsplit prime $\ell \mid N^{+}$, so \eqref{eq:gen-Heegner} is automatic.

Let $\scL_{\frp}^{\BDP}(f/K)\in\Lambda_{\ac}^{\ur}$ be the square-root $p$-adic $L$-function of \cite[Theorem~4.1]{CLW}. It is the anticyclotomic $p$-adic $L$-function introduced by Bertolini--Darmon--Prasanna in \cite{BDP} with the refinement and generalization in \cite[Section 2.5]{CastellaHsieh} as in \cite[Theorem 2.1.3, Remark 2.1.4]{Cas25}. We record its interpolation formula as follows as \cite[Theorem 4.1]{CLW}. \footnote{The $p$-adic $L$-function $\calL_{p}^{\BDP}(f/K)$ in Theorem \ref{thm:BDP} is the $p$-adic $L$-function $\scL_{\barv}^{\BDP}(f/K)^{2}$, and $\scL_{\barv}^{\BDP}(f/K)$ in \textit{op.cit.} is called the \emph{``square-root'' $p$-adic $L$-function}.}

\begin{theorem}[Anticyclotomic $p$-adic $L$-function] \label{thm:BDP}
    There is an element $\calL_{p}^{\BDP}(f/K) \in \Lambda_{\ac}^{\ur}$ such that for every character $\xi$ of $\Gamma_{\ac}$ crystalline at both $\frp$ and $\barfrp$ corresponding to a Hecke character of $K$ of infinite type $(-j,j)$ for some $j \geq k/2$ with $j \equiv 0 \pmod{p-1}$, we have
    \begin{multline} \label{eq:BDP-interpolation}
        \calL_p^{\BDP}(f/K)(\xi) =\left(\frac{\Omega_p}{\Omega_\infty}\right)^{4j}(r+j-1)!(j-r)!\left(\frac{2\pi}{\sqrt{D_K}}\right)^{2j-1}\\
   \times \left(1-a_p(f)\xi(\frp)p^{-r}+\xi(\frp)^2p^{-1}\right)^2 L(f/K,\xi,r).
    \end{multline}
The CM periods $(\Omega_{\infty}, \Omega_{p}) = (2 \pi i \cdot \Omega_K, \Omega_p) \in \CC^{\times} \times \calO_{\CC_p}^{\times}$ are as in \cite[Section~2.5]{CastellaHsieh}.
\end{theorem}

We \emph{define} the preparatory \emph{Greenberg's $p$-adic $L$-function} for $f$ over $K$ as
\[
\widetilde{\calL_p^{\Gr}}(f/K)^{\eta} := \Tw_{r-1}( \calL_{p}^{\two}(\cBF^{\alpha, \bfh_{\frp},\eta})) \in \scR^{\ur}.
\]

The projection of $\widetilde{\calL_p^{\Gr}}(f/K)$ to the anticyclotomic line recovers the anticyclotomic $p$-adic $L$-function of $f$ over $K$. Before proving this, we recall in \cite[Proposition 5.17]{BSTW} the projection map $\phi_{\ac}: \scR^{\ur} \twoheadrightarrow \Lambda_{\ac}^{\ur}$ induced by the homomorphism
\[
\Gamma_{\frp} \times \Gamma_{\cyc} \twoheadrightarrow \Gamma_{\ac}, \quad \gamma_{\frp} \mapsto \gamma_{\ac}^{2}, \, \gamma_{\cyc} \mapsto \gamma_{\ac}^{p^{h_{p}}}.
\]
So $T_{\frp} =\gamma_{\frp} - 1 \mapsto \gamma_{\ac}^2-1 =(1+T_{\ac})^2-1$.

\begin{proposition} \label{prop:anticyclotomic_descent_0}
    Let $I_{\cyc}^{\ur}$ be the kernel of the natural projection $\phi_{\ac}: \scR^{\ur} \twoheadrightarrow \Lambda_{\ac}^{\ur}$. Suppose that \eqref{eq:gen-Heegner} holds. Then 
    \[
    (\widetilde{\calL_p^{\Gr}}(f/K)^{\omega^{1-r}} \bmod{I_{\cyc}^{\ur}}) = ((1+T_{\ac})^{2} - 1)(\calL_{p}^{\BDP}(f/K)^{\jmath^{\ac}}),
    \]
    where $(-)^{\jmath^{\ac}}$ is the image under the involution $\Lambda_{K}^{\ur} \rightarrow \Lambda_{K}^{\ur}$ induced by $\gamma_{\ac} \mapsto \gamma_{\ac}^{-1}$.
\end{proposition}
\begin{proof}
The proof compares the interpolation formulas in Theorems \ref{thm:ERLIIint} and \ref{thm:BDP}. It proceeds exactly as in \cite[Proposition 4.4]{CLW}\footnote{This is also shown in \cite[Proposition 5.28]{BSTW} which follows from \cite[Proposition 5.17]{BSTW}, under the Heegner hypothesis, a more restrictive condition than the generalized Heegner hypothesis \eqref{eq:gen-Heegner}.}, but we provide the details for the reader's convenience. We use the following parameters:
\[
m = 2j, \quad n = r+j-1, \quad \zeta = 1
\]
with $j \geq r$ and $j \equiv 0 \pmod{p-1}$. In this setup, the second explicit reciprocity law (\eqref{eq:ERLII}) can be simplified further. We look at the formula part by part.

(1) We first treat the denominator of \eqref{eq:ERLII}. After composing with the quotient $\Gamma_{K} \twoheadrightarrow \Gamma_{\frp}$ and accounting for the extra twist $\Tw_{r-1}$, the character $\chi \in \Xi^{\two}$ is $\chi_1\epsilon^{j}$, corresponding to a Hecke character $\xi$ of infinite type $(-j,j)$. In particular, $\xi(\frp) = p^{j} \chi_1(\barfrp)^{-1}$. 

For the Euler factors in the denominator, we have
\begin{align*}
(1-p^{2j-1} \chi_1(\barfrp)^{-2})(1-p^{2j}\chi_1(\barfrp)^{-2}) &= (1-\xi(\frp)^{2} p^{-1})(1-\xi(\frp)^{2}) \\ 
&= (1-\xi^{-1}(\barfrp)^{2} p^{-1})(1-\xi(\frp)^{2}),
\end{align*}
where the second equality follows from $\xi^{c} = \xi^{-1}$ since $\xi$ is an anticyclotomic character. Therefore, the entire denominator of \eqref{eq:ERLII} becomes
\[
\pi^{2j+1} (-i)^{2j+1+k} 2^{4j+1} (1-\xi^{-1}(\barfrp)^{2} p^{-1})(1-\xi(\frp)^{2}) \lrangle{\bfh_{\frp, \chi_1}^{\circ},\bfh_{\frp, \chi_1}^{\circ}}_{D_{K}}.
\]
To simplify this expression further, we apply Hida's adjoint $L$-value formula (see \cite[Theorem 7.3]{HT93}) to obtain
\[
\lrangle{\bfh_{\frp, \chi_1}^{\circ},\bfh_{\frp, \chi_1}^{\circ}}_{D_{K}} \sim_{p} \dfrac{h_K}{w_K} \dfrac{(2j)!}{2^{4j-1}\pi^{2j+1}} L(1, \xi^{2}),
\]
where we have used the fact that $\xi/\xi^{c} = \xi^{2}$. Substituting the interpolation formula \eqref{eq:Katz_interpolation} for $\calL_{\barfrp}^{\Katz}(\xi^{2}\bfN^{-1})$, of infinite type $(1-2j, 1+2j)$, into the expression above gives
\[
\text{denominator of \eqref{eq:ERLII}} \sim_p \dfrac{\dfrac{h_K}{w_K} \cdot \calL_{\barfrp}^{\Katz}(\xi^{2}\bfN^{-1}) \cdot (-i)^{2j+1-k} 2^{4}}{\left(\dfrac{\Omega_{p}}{\Omega_{\infty}} \right)^{4j} \left(\dfrac{2\pi}{\sqrt{D_{K}}}\right)^{2j-1}}
\]
By the functional equation for Katz's $p$-adic $L$-function \eqref{eq:Katz_functional_equation}, $\calL_{\barfrp}^{\Katz}(\xi^{2}\bfN^{-1}) = \calL_{\barfrp}^{\Katz}(\xi^{-2})$; the latter is exactly the involuted Katz $p$-adic $L$-function $\calL_{\barfrp}^{\Katz, \frp}(K)^{\jmath_{\ac}}$ evaluated at $\xi$. Thus
\[
\boxed{\text{denominator of \eqref{eq:ERLII}}} \sim_p \dfrac{\dfrac{h_K}{w_K} \cdot \phi_{\xi}\left(\calL_{\barfrp}^{\Katz, \frp}(K)^{\jmath_{\ac}}\right)}{\left(\dfrac{\Omega_{p}}{\Omega_{\infty}} \right)^{4j} \left(\dfrac{2\pi}{\sqrt{D_{K}}}\right)^{2j-1}}
\]
Here we have removed the term $(-i)^{2j+1-k} 2^{4}$ since it is a fixed $p$-adic unit. Indeed, recall that $j \equiv 0 \pmod{p-1}$, and hence $j$ is an even number. This gives $(-i)^{2j+1-k} 2^{4} = (-i)^{1-k} 2^{4} \in (\calO^{\ur})^{\times}$. \footnote{If $p\equiv1\pmod4$, then $i \in \QQ_p$. If $p\equiv3\pmod4$, adjoining $i$ over $\QQ_p$ gives an unramified quadratic extension of $\QQ_p$, so $i$ lies in the coefficient ring after the unramified extension already used in $\Lambda_{K}^{\ur}$.}

(2) Second, we simplify the Euler factor in the numerator:
\begin{align*}
\boxed{\calE_{p, \eta}(\bfh_{\frp}, f,\chi)} &= (1-p^{j-r}\alpha \chi_1(\barfrp)^{-1})^2 (1-p^{j-r}\beta \chi_1(\barfrp)^{-1})^2 \\ 
&= (1-\xi(\frp)\alpha p^{-r})^{2} (1-\xi(\frp)\beta p^{-r})^{2} \\
&= (1- a_p(f)\xi(\frp)p^{-r}+\xi(\frp)^{2}p^{-1})^{2},
\end{align*}
since $\alpha+\beta = a_p(f)$ and $\alpha\beta = p^{2r-1}$. Here we have used the assumption that $\eta = \omega^{1-r}$, so $\eta = \omega^{-n}$. 

(3) Finally, we treat the $L$-value. Since $\eta = \omega^{1-r}$, we have
\[
\omega^{n} \eta = \omega^{n+1-r} = \omega^{j} = 1 
\]
since $j \equiv 0 \pmod{p-1}$. Therefore, 
\begin{align*}
\boxed{L(\bfh_{\frp, \chi_1}^{\circ}, f \otimes \omega^{n}\eta \psi_{\zeta}^{-1}, n+1)} &= L(\bfh_{\frp, \chi_1}^{\circ}, f \otimes \psi_{\zeta}^{-1}, j+r) \\
&= \sum_{\mathfrak{a}, \, (\mathfrak{a},\mathfrak{f}_{\psi_\chi})=1}
a_f(N(\mathfrak{a})) \psi_\chi(x_{\mathfrak{a}}) N(\mathfrak{a})^{-(j+r)} = L(f, \psi_\chi, j+r)
\end{align*}
The latter is $L(f/K, \xi, r)$. Putting the three boxed equations together into \eqref{eq:ERLII}, we obtain the result
    \[
    (\jmath_{\ac} \circ \phi_{\ac})(\widetilde{\calL_p^{\Gr}}(f/K)^{\omega^{1-r}}) = ((1+T_{\ac})^{2} - 1)(\calL_{p}^{\BDP}(f/K)),
    \]
by comparing with the interpolation formula in Theorem \ref{thm:BDP}. Applying the involution $\jmath^{\ac}: \Lambda_{\ac}^{\ur} \rightarrow \Lambda_{\ac}^{\ur}$ gives the desired result.
\end{proof}

\subsection{Local properties of the Beilinson--Flach elements at $p$}
We are now ready to establish the analogue of \cite[Theorem 5.20]{BSTW}.

\begin{theorem} \label{thm:BF-descent}
The natural injection 
\[
\rmH^1_{\rel,\ordi}(\QQ, T \hotimes(\TT_{\bfh_\frp}^{\square})\hotimes \Lambda_{\cyc}^\iota) \hookrightarrow \rmH^1_{\rel,\ordi}(\QQ, T \hotimes(\widetilde{\TT}_{\bfh_\frp}^{\square})\hotimes \Lambda_{\cyc}^\iota) 
\]
contains $\cBF^{\alpha,\bfh_{\frp}, \eta}$ in its image.
\end{theorem}

The injectivity of the map in Theorem \ref{thm:BF-descent} follows from Lemma \ref{lem:vanishing_H0}. The proof proceeds as in \cite[Theorem 5.20]{BSTW}; we explain how it generalises to our setting. \footnote{Its relation with \cite[Section 5.4]{BSTW} requires a brief clarification. In that section, the proofs of Theorems 5.19 and 5.20 are intertwined. The reason is that the property \textbf{(Ind)} (explicitly stated in \cite[Section 4.2.9]{BSTW}) for the canonical CM family $\bfh_{\frp}$, which is the content of \cite[Theorem 5.19]{BSTW}, had not yet been established. The second explicit reciprocity law in \cite{BSTW} is used there to rule out the failure of \textbf{(Ind)} and hence to prove Theorem 5.19. In the present article, we take \cite[Theorem~5.19]{BSTW} as an established input (and the same occurs in \cite{CLW}), recorded here in Theorem \ref{thm:CM-lattices} with a conclusion concerning only the integral Galois representation attached to the canonical CM Hida family and independent of the modular form $f$ considered here. Once \textbf{(Ind)} is known, the intermediate induced lattice $\II$ of \cite[Lemma~4.18]{BSTW} coincides with the lattice $\TT_{\bfh_{\frp}}^{\square}$, and the proof of \cite[Theorem 5.20]{BSTW} carries over. Thus there is no circularity.} Before giving the proof, we record some consequences of Shapiro's lemma. For simplicity, for any lattice $M \in \{ \TT_{\bfh_{\frp}}, \tilTT_{\bfh_{\frp}}, \TT_{\bfh_{\frp}}^{\square}, \TT_{\bfh_{\frp}}^{\square}, \TT_{\bfh_{\frp}}^{\star}, \tilTT_{\bfh_{\frp}}^{\star}\}$, we set $\scA(M) := T \hotimes_{\ZZ_p} M \hotimes_{\ZZ_p} \Lambda_{\cyc}^{\iota}$ and 
    \begin{align}
    \scS(M) &:= \rmH^{1}_{\ord, \rel}(\QQ, \scA(M)) \notag \\
    &:= \{c \in \rmH^{1}(\QQ, \scA(M)): \loc_{p}(c)  = 0 \in \rmH^{1}(\QQ_p, \Fil^{-}T \hotimes M/M_{\frp} \hotimes \Lambda_{\cyc}^{\iota}) \} \label{eq:define_Q_relord}
    \end{align}
    whenever $M_{\frp}$ makes sense. By Theorem \ref{thm:CM-lattices}(2), Shapiro's lemma gives
    \begin{equation} \label{eq:shapiro_global}
    \rmH^{1}(\QQ, \scA(\TT_{\bfh_{\frp}}^{\square})) \simeq \rmH^{1}(K, T \hotimes \Lambda_{\frp}^{\iota} \hotimes \Lambda_{\cyc}^{\iota})
    \end{equation}
    and
    \begin{equation} \label{eq:shapiro_local}
        \rmH^{1}(\QQ_p, \scA(\TT_{\bfh_{\frp}}^{\square})) \simeq \rmH^{1}(K_{\frp}, T \hotimes \TT_{\bfh_{\frp}, \frp} \hotimes \Lambda_{\cyc}^{\iota}) \oplus \rmH^{1}(K_{\barfrp}, T \hotimes \TT_{\bfh_{\frp}, \barfrp} \hotimes \Lambda_{\cyc}^{\iota})
    \end{equation}
    with compatible localisation maps $\loc_{p}$, $\loc_{\frp}$ and $\loc_{\barfrp}$. Moreover, consider the map defining the local condition at $p$:
    \[
    \rmH^{1}(\QQ_p, T \hotimes \TT_{\bfh_{\frp}}^{\square} \hotimes \Lambda_{\cyc}^{\iota}) \rightarrow \rmH^{1}(\QQ_p, \Fil^{-} T \hotimes (\TT_{\bfh_{\frp}}^{\square}/\TT_{\bfh_{\frp}, \frp}) \hotimes \Lambda_{\cyc}^{\iota}) \simeq \rmH^{1}(\QQ_p, \Fil^{-} T \hotimes \TT_{\bfh_{\frp}, \barfrp} \hotimes \Lambda_{\cyc}^{\iota}).
    \]
    Under local Shapiro's lemma \eqref{eq:shapiro_local}, this gives exactly the relaxed condition at $\frp$ and the ordinary condition at $\barfrp$. Consequently, global Shapiro's lemma \eqref{eq:shapiro_global} restricts to an isomorphism
    \begin{equation} \label{eq:shapiro_final}
    \rmH^{1}_{\rel, \ord}(\QQ, \scA(\TT_{\bfh_{\frp}}^{\square})) \simeq \rmH^{1}_{\rel,\ord}(K, T \hotimes \TT_{\bfh_{\frp},\frp} \hotimes \Lambda_{\cyc}^{\iota}) \simeq \rmH^{1}_{\rel,\ord}(K, T \hotimes \Lambda_{\frp}^{\iota} \hotimes \Lambda_{\cyc}^{\iota}),
    \end{equation}
    where the last isomorphism is induced upon fixing an isomorphism $\TT_{\bfh_{\frp},\frp} \simeq \Lambda_{\frp}^{\iota}$ with $\TT_{\bfh_{\frp},\frp}/T_{\frp}\TT_{\bfh_{\frp},\frp} \simeq \ZZ_p$.

    We are now ready to present the proof.
\begin{proof}[Proof of Theorem \ref{thm:BF-descent}]
    By Proposition \ref{prop:local_BF}(1), the Beilinson--Flach element $\cBF^{\alpha, \bfh_{\frp}, \eta}$ satisfies the local condition defining $\scS(\tilTT_{\bfh_{\frp}}^{\square})$ in \eqref{eq:define_Q_relord}. Let $b \in \scS(\tilTT_{\bfh_{\frp}}^{\star})$ denote its image under the map induced by $\tilTT_{\bfh_{\frp}}^{\square} \rightarrow \tilTT_{\bfh_{\frp}}^{\star}$.

    By Theorem \ref{thm:CM-lattices}, we have $T_{\frp}(\tilTT_{\bfh_{\frp}}^{\star}) \subseteq \TT_{\bfh_{\frp}}^{\square}$. We then write $b^{\prime} := T_{\frp} \cdot b \in \rmH^{1}(\QQ, \scA(\TT_{\bfh_{\frp}}^{\square}))$. We claim that $b^{\prime} \in \scS(\TT_{\bfh_{\frp}}^{\square})$. Indeed, the inclusion $\TT_{\bfh_{\frp}}^{\square}/\TT_{\bfh_{\frp}, \frp} \hookrightarrow \tilTT_{\bfh_{\frp}}^{\star}/\tilTT_{\bfh_{\frp}, \frp}$, whose cokernel is $\tilTT_{\bfh_{\frp}}^{\star}/\TT_{\bfh_{\frp}}^{\square}$, gives an exact sequence
    \begin{multline*}
    \rmH^{0}(\QQ_p, \Fil^{-} T \hotimes (\tilTT_{\bfh_{\frp}}^{\star}/\TT_{\bfh_{\frp}}^{\square}) \hotimes \Lambda_{\cyc}^{\iota}) \\ \rightarrow \rmH^{1}(\QQ_p, \Fil^{-} T \hotimes (\TT_{\bfh_{\frp}}^{\square}/\TT_{\bfh_{\frp}, \frp}) \hotimes \Lambda_{\cyc}^{\iota}) \rightarrow \rmH^{1}(\QQ_p, \Fil^{-} T \hotimes (\tilTT_{\bfh_{\frp}}^{\star}/\tilTT_{\bfh_{\frp}, \frp}) \hotimes \Lambda_{\cyc}^{\iota})
    \end{multline*}
    By Lemma \ref{lem:vanishing_H0}, the leftmost term vanishes. The image of $\loc_{p}(b^{\prime})$ in the rightmost term vanishes because $b \in \scS(\tilTT_{\bfh_{\frp}}^{\star})$. This then implies that $b^{\prime} \in \scS(\TT_{\bfh_{\frp}}^{\square})$.

    Next we consider the coefficient exact sequence
    \[
    0 \rightarrow \scA(\TT_{\bfh_{\frp}}^{\square}) \xrightarrow{\times T_{\frp}} \scA(\TT_{\bfh_{\frp}}^{\square}) \rightarrow T \hotimes (\TT_{\bfh_{\frp}}^{\square}/T_{\frp}\TT_{\bfh_{\frp}}^{\square}) \hotimes \Lambda_{\cyc}^{\iota} \rightarrow 0.
    \]
    By Lemma \ref{lem:vanishing_H0} together with the identification \eqref{eq:shapiro_final}, we have an injection (which is \cite[Equation (5.13)]{BSTW}):
    \[
    \varrho: \scS(\TT_{\bfh_{\frp}}^{\square})/T_{\frp} \scS(\TT_{\bfh_{\frp}}^{\square}) \hookrightarrow \rmH^{1}_{\rel, \ord}(K, T \hotimes \Lambda_{\cyc}^{\iota}).
    \]
    By Proposition \ref{prop:rank}(2), we see that $\rmH^{1}_{\rel, \ord}(K, T \hotimes \Lambda_{\cyc}^{\iota})$ is free of rank one over $\Lambda_{\cyc}$. Let $\barb^{\prime}$ denote the image of $b^{\prime}$ modulo $T_{\frp}$.

    Project $\loc_{p}(b)$ to the $\Fil^{-} T$-quotient $\rmH^1(\QQ_p,\Fil^{-} T \hotimes(\widetilde{\TT}_{\bfh_\frp}) \hotimes \Lambda_{\cyc}^\iota)$. Proposition \ref{prop:local_BF}(3) says that the image belongs to $\rmH^{1}(\QQ_p, \Fil^{-}T \hotimes \Fil^{+} \tilTT_{\bfh_{\frp}} \hotimes \Lambda_{\cyc}^{\iota})$. It follows from Theorem \ref{thm:CM-lattices} and Ohta's filtration that $\Fil^{+} \tilTT_{\bfh_{\frp}} = \Fil^{+} \TT_{\bfh_{\frp}} =  \TT_{\bfh_{\frp}, \frp}$, and hence 
    \[
    \loc_{p}(b) \in \rmH^{1}(\QQ_p, \Fil^{-}T \hotimes \TT_{\bfh_{\frp}, \frp} \hotimes \Lambda_{\cyc}^{\iota})
    \]
    already. Therefore $\loc_{p}(b^{\prime}) = T_{\frp} \cdot \loc_{p}(b)$ is divisible by $T_{\frp}$ inside the local cohomology group with coefficients in $\TT_{\bfh_{\frp}, \frp}$. Hence $\loc_{p}(b^{\prime}) \equiv 0 \pmod{T_{\frp}}$, which implies that 
    \[
    \Cole_{f, \frp} \left( \loc_{\frp}(\varrho(\barb^{\prime})) \right) = 0,
    \]
    where $\Cole_{f, \frp}$ is the classical Coleman map of $f$ transferred to $K$, defined in \cite[Section 3.1.2]{BSTW}. (This is \cite[Equation (5.14)]{BSTW}.) By the same argument as in Corollary \ref{coro:cyclo_nonvanishing}, the composition
    \[
    \rmH^{1}_{\rel, \ord}(K, T \hotimes \Lambda_{\cyc}^{\iota}) \xrightarrow{\loc_{\frp}} \rmH^{1}(K_{\frp}, \Fil^{-} T \hotimes \Lambda_{\cyc}^{\iota}) \xrightarrow{\Cole_{f,\frp}} \Frac(\Lambda_{\cyc})
    \]
    is nonzero. Since $\rmH^{1}_{\rel, \ord}(K, T \hotimes \Lambda_{\cyc}^{\iota})$ is free of rank one over $\Lambda_{\cyc}$, the composition is injective. This implies that $\varrho(\barb^{\prime}) = 0 \in \rmH^{1}_{\rel, \ord}(K, T \hotimes \Lambda_{\cyc}^{\iota})$, and hence $\barb^{\prime} = 0 \in \scS(\TT_{\bfh_{\frp}}^{\square})/T_{\frp} \scS(\TT_{\bfh_{\frp}}^{\square})$ as well since $\varrho$ is injective. In other words, there exists $c \in \scS(\TT_{\bfh_{\frp}}^{\square})$ such that $b^{\prime} = T_{\frp} \cdot c$. As $b^{\prime} = T_{\frp} \cdot b$ already, this implies that $T_{\frp} \cdot b = T_{\frp} \cdot c$ holding inside $\scS(\tilTT_{\bfh_{\frp}}^{\star})$.

    We next note that multiplication by $T_{\frp}$ is injective on $\rmH^{1}(\QQ, \scA(\tilTT_{\bfh_{\frp}}^{\star}))$: indeed, consider the cohomology sequence induced by multiplication by $T_{\frp}$:
    \[
    \rmH^{0}(\QQ, T \hotimes (\tilTT_{\bfh_{\frp}}^{\star}/T_{\frp}\tilTT_{\bfh_{\frp}}^{\star}) \hotimes \Lambda_{\cyc}^{\iota}) \rightarrow \rmH^{1}(\QQ, T \hotimes \tilTT_{\bfh_{\frp}}^{\star} \hotimes \Lambda_{\cyc}^{\iota}) \rightarrow \rmH^{1}(\QQ, T \hotimes \tilTT_{\bfh_{\frp}}^{\star} \hotimes \Lambda_{\cyc}^{\iota}).
    \]
    Again Lemma \ref{lem:vanishing_H0} implies that the leftmost term vanishes, and hence the multiplication-by-$T_{\frp}$ map is indeed injective. This implies that $b = c \in \scS(\tilTT_{\bfh_{\frp}}^{\square})$. Since $c \in \scS(\TT_{\bfh_{\frp}}^{\square})$, this gives $b \in \scS(\TT_{\bfh_{\frp}}^{\square})$.
    
    Finally, to conclude that $\cBF^{\alpha, \bff,\eta} \in \scS(\TT_{\bfh_{\frp}}^{\square})$, it remains to show that the map 
    \[
    \rmH^{1}(\QQ, T \hotimes \tilTT_{\bfh_{\frp}}^{\square} \hotimes \Lambda_{\cyc}^{\iota}) \rightarrow \rmH^{1}(\QQ, T \hotimes \tilTT_{\bfh_{\frp}}^{\star} \hotimes \Lambda_{\cyc}^{\iota})
    \]
    induced by the inclusion $\tilTT_{\bfh_{\frp}}^{\square} \hookrightarrow \tilTT_{\bfh_{\frp}}^{\star}$ is injective. Indeed, the inclusion induces the following exact sequence
    \[
    \rmH^{0}(\QQ, T \hotimes (\tilTT_{\bfh_{\frp}}^{\star}/\tilTT_{\bfh_{\frp}}^{\square}) \hotimes \Lambda_{\cyc}^{\iota}) \rightarrow \rmH^{1}(\QQ, T \hotimes \tilTT_{\bfh_{\frp}}^{\square} \hotimes \Lambda_{\cyc}^{\iota}) \rightarrow \rmH^{1}(\QQ, T \hotimes \tilTT_{\bfh_{\frp}}^{\star} \hotimes \Lambda_{\cyc}^{\iota}),
    \]
    and Lemma \ref{lem:vanishing_H0} implies that the leftmost term vanishes, as desired. The theorem is therefore proved.
\end{proof}

\subsection{Explicit reciprocity law (II, \textit{bis})}

By Theorem \ref{thm:BF-descent}, the class $\cBF^{\alpha, \bfh_{\frp}, \eta}$ already lies in $\rmH^1_{\rel,\ordi}(\QQ, T \hotimes(\TT_{\bfh_\frp}^{\square})\hotimes \Lambda_{\cyc}^\iota)$. Therefore, under the integral logarithm map $\Log_{\barfrp}^{\inte}$, the image $\widetilde{\calL_{p}^{\two}}(\cBF^{\alpha, \bfh_{\frp}, \eta})$ is divisible by $T_{\frp}$. It is therefore well-defined to set
\[
\calL_{p}^{\two}(\cBF^{\alpha, \bfh_{\frp}, \eta}) := \dfrac{\widetilde{\calL_{p}^{\two}}(\cBF^{\alpha, \bfh_{\frp}, \eta})}{T_{\frp}} \in \scR^{\ur},
\]
and to define the \emph{Greenberg $p$-adic $L$-function} for $f$ over $K$ by
\[
\calL_p^{\Gr}(f/K)^{\eta} := \Tw_{r-1}( \calL_{p}^{\two}(\cBF^{\alpha, \bfh_{\frp},\eta})) \in \scR^{\ur}.
\]
The following results follow from Theorem \ref{thm:ERLIIint} and Proposition \ref{prop:anticyclotomic_descent_0} directly.

\begin{theorem}[Explicit Reciprocity Law II]\label{thm:ERLIIint_bis}
The element $\calL_p^{\two}(\cBF^{\alpha, \bfh_{\frp}})\in \scR^\ur$ has the following interpolation property: for $\chi\in \Xi^{\two}$, 
\begin{multline*} \label{eq:ERLII_cusp}
    \phi_{\chi} \left( \calL_{p}^{\two}(\cBF^{\alpha, \bfh_{\frp}, \eta}) \right) = \dfrac{\calE_{p, \eta}(\bfh_{\frp}, f,\chi)}{(1-p^{m-1} \chi_1(\barfrp)^{-2})(1-p^{m}\chi_1(\barfrp)^{-2})} \dfrac{(-1)^{n+1} (n!)(n+1-k)!}{\pi^{2n+3-k} (-i)^{m+1-k} 2^{2n+m+3-k}} \\ \times \dfrac{h_{K}}{w_K} \cdot \phi_{\chi}( \calL_{\barfrp}^{\Katz}(K)) \cdot \dfrac{L(\bfh_{\frp, \chi_1}^{\circ}, f \otimes \omega^{n}\eta \psi_{\zeta}^{-1}, n+1)}{\lrangle{\bfh_{\frp, \chi_1}^{\circ},\bfh_{\frp, \chi_1}^{\circ}}_{D_{K}}},
\end{multline*}
where $\calE_{p,\eta}(\bfh_{\frp}, f, \chi)$ is the Euler factor defined in \eqref{eq:Euler-II}.
\end{theorem}

\begin{proposition} \label{prop:anticyclotomic_descent}
    Let $I_{\cyc}^{\ur}$ be the kernel of the natural projection $\phi_{\ac}: \scR^{\ur} \twoheadrightarrow \Lambda_{\ac}^{\ur}$. Suppose that \eqref{eq:gen-Heegner} holds. Then 
    \[
    (\calL_p^{\Gr}(f/K)^{\omega^{1-r}} \bmod{I_{\cyc}^{\ur}}) = (\calL_{p}^{\BDP}(f/K)^{\jmath_{\ac}}).
    \]
\end{proposition}

Having established the anticyclotomic descent theorem for the anticyclotomic $p$-adic $L$-function, we obtain the following $\mu$-invariant vanishing result of Hsieh.

\begin{theorem}[Hsieh] \label{thm:mu-vanishing}
Assume that conditions \eqref{eq:gen-Heegner} and \eqref{eq:irred} hold. Then $\mu\bigl(L_p^{\BDP}(f/K)\bigr)=0$, and hence $\mu\bigl(\calL_p^{\Gr}(f/K)\bigr)=0$.\end{theorem}

\begin{proof}
By \cite[Theorem~B]{Hsieh2014}, we have $\mu(\calL_{\frp}^{\BDP}(f/K))=0$, and hence $\mu(\calL_{\frp}^{\BDP}(f/K)^{\jmath_{\ac}})=0$. By the anticyclotomic descent theorem (Proposition \ref{prop:anticyclotomic_descent}), the projection of $\calL_{p}^{\Gr}(f/K)$ to $\Lambda_{K}^{\ac,\ur}$ generates the same ideal as $\calL_{\frp}^{\BDP}(f/K)^{\jmath_{\ac}}$; the result follows.
\end{proof}

%% file: 04_zeta_morphism.tex
\section{Zeta elements over imaginary quadratic fields and main conjectures} \label{sec:zeta_main_conjecture}

\subsection{Zeta element of $f$ over $K$} \label{sec:zeta_definition}
We are finally ready to define the zeta element of $f$ over $K$. It follows from Theorem \ref{thm:BF-descent} that
\[
\cBF^{\alpha, \bfh_{\frp}, \eta} \in \rmH^1_{\rel,\ordi}(\QQ, T \hotimes(\TT_{\bfh_\frp}^{\square})\hotimes \Lambda_{\cyc}^\iota).
\]
By Theorem \ref{thm:CM-lattices}(2), Shapiro's lemma gives the following isomorphism
\begin{equation} \label{eq:shapiro}
\mathrm{Spr}: \rmH^1_{\rel,\ordi}(\QQ, T \hotimes(\TT_{\bfh_\frp}^{\square})\hotimes \Lambda_{\cyc}^\iota) \xrightarrow{\sim} \rmH^{1}_{\rel,\ordi}(K, T \hotimes \Lambda_{\frp}^{\iota} \hotimes \Lambda_{\cyc}^{\iota})  \xrightarrow{\gamma_{\frp} \mapsto \gamma_{\frp}^{-1}} \rmH^{1}_{\rel,\ordi}(K, T \hotimes \Lambda_{\frp} \hotimes \Lambda_{\cyc}^{\iota})
\end{equation}
Recall that in Section \ref{sec:more-iwasawa-algebra}, we defined an isomorphism
\begin{equation} \label{eq:theta_isomorphism}
\theta: \Gamma_{K} \xrightarrow{\sim} \Gamma_{\frp} \times \Gamma_{\cyc} \xrightarrow[\sim]{\gamma_{\frp} \mapsto \gamma_{\frp}^{-1}} \Gamma_{\frp} \times \Gamma_{\cyc}.
\end{equation}
This therefore determines an isomorphism $\theta: \Lambda_{K} \xrightarrow{\sim} \scR$ and $\theta^{\ur}: \Lambda_{K}^{\ur} \xrightarrow{\sim} \scR^{\ur}$. The isomorphism \eqref{eq:theta_isomorphism} also induces $\theta: T \hotimes \Lambda_{K}^{\iota} \xrightarrow{\sim} T \hotimes \Lambda_{\frp} \hotimes \Lambda_{\cyc}^{\iota}$, which is compatible with $\theta$. We thus obtain an identification
\[
\theta: \rmH^{1}_{\rel, \ordi}(K, \bfT_{K}) = \rmH^{1}_{\rel, \ordi}(K, T \hotimes \Lambda_{K}^{\iota}) \rightarrow \rmH^{1}_{\rel, \ord}(K, T \hotimes \Lambda_{\frp} \hotimes \Lambda_{\cyc}^{\iota}).
\]

\begin{definition}[Zeta element of $f$ over $K$]
\label{def:zeta}
The image of $\cBF^{\alpha, \bfh_{\frp}, \eta}$ under the composite map $\theta^{-1} \circ \mathrm{Spr}$
\[
\rmH^1_{\rel,\ordi}(\QQ, T \hotimes(\TT_{\bfh_\frp}^{\square})\hotimes \Lambda_{\cyc}^\iota) \xrightarrow{\sim} \rmH^{1}_{\rel,\ordi}(K, T \hotimes \Lambda_{\frp} \hotimes \Lambda_{\cyc}^{\iota}) \rightarrow \rmH^{1}_{\rel, \ordi}(K, \bfT_{K})
\]
is the \emph{zeta element of $f$ over $K$}, denoted by $\calZ^{\eta}(f/K)$. For a quotient $\Lambda_K \twoheadrightarrow\Lambda_\dagger$ with $\dagger \in \{\cyc, \ac\}$, write $\calZ^{\eta}_\dagger(f/K)$ for its image in $H^1_{\rel,\ordi}(K, \bfT_{\dagger})$. We often omit the superscript $\eta$ when it is fixed and clear from the context.
\end{definition}

The explicit reciprocity laws for the zeta element $\calZ(f/K)$ follow from those for the Beilinson--Flach elements in the previous section. Let $\Cole(f/K)$ be the composite
\begin{align*} \label{eq:col_f_K}
\Cole(f/K): \rmH^{1}(K_{\frp}, \Fil^{-}\bfT_{K}) &\xrightarrow[\sim]{\mathrm{Spr}} \rmH^{1}(\QQ_p, \Fil^{-}T \hotimes (\Fil^{+} \widetilde{\TT}_{\bfh_\frp}) \hotimes \Lambda_{\cyc}^\iota) \\ 
&\xrightarrow{\Cole_{\frp}} \scR \xrightarrow[\sim]{\theta^{-1}} \Lambda_{K}.
\end{align*}

Following \cite[Section~5.5.1]{BSTW}, the inclusion $\TT_{\bfh_{\frp}}^{\square} = \TT_{\bfh_{\frp}}^{\star} \hookrightarrow
\widetilde{\TT}_{\bfh_{\frp}}^{\star}$ induces an isomorphism 
\begin{equation}\label{eq:minus-lattice-identity}
\Fil^{-}\TT_{\bfh_{\frp}} = T_{\frp}\cdot \Fil^{-}\widetilde{\TT}_{\bfh_{\frp}}^{\star} \subseteq \Fil^{-}\widetilde{\TT}_{\bfh_{\frp}}^{\star} = \widetilde{\TT}_{\bfh_{\frp}}^{\star}
/\widetilde{\TT}_{\bfh_{\frp},\frp}
\end{equation}
Hence, under the inclusion $\rmH^{1}(\QQ_p, \Fil^{+}T \hotimes (\Fil^{-} \TT_{\bfh_\frp}) \hotimes \Lambda_{\cyc}^\iota) \hookrightarrow \rmH^{1}(\QQ_p, \Fil^{+}T \hotimes (\Fil^{-} \widetilde{\TT}_{\bfh_\frp}^{\star}) \hotimes \Lambda_{\cyc}^\iota)$, we have
\begin{equation}\label{eq:integral-log-divisibility}
\Log_{\barfrp}^{\inte}\left( \rmH^{1}\left(
\QQ_p, \Fil^{+}T \hotimes\Fil^{-}\TT_{\bfh_{\frp}} \hotimes\Lambda_{\cyc}^{\iota} \right) \right) \subseteq T_{\frp}\scR^{\ur}.
\end{equation}
It is therefore well-defined to introduce the normalized logarithm map
\[
\Log_{\barfrp}^{\mathrm{norm}} := T_{\frp}^{-1} \cdot \Log_{\barfrp}^{\inte}: \rmH^{1}\left( \QQ_p, \Fil^{+}T \hotimes\Fil^{-}\TT_{\bfh_{\frp}} \hotimes\Lambda_{\cyc}^{\iota} \right) \rightarrow \scR^{\ur}
\]
on the cohomology group appearing in
\eqref{eq:integral-log-divisibility}. We then define $\Log(f/K)$ to be the composite
\begin{align*}
\Log(f/K): \rmH^{1}(K_{\barfrp}, \Fil^{+}\bfT_{K}) &\xrightarrow[\sim]{\mathrm{Spr}} \rmH^{1}(\QQ_p, \Fil^{+}T \hotimes (\Fil^{-} \TT_{\bfh_\frp}) \hotimes \Lambda_{\cyc}^\iota) \\
&\xrightarrow{\Log_{\bar{\frp}}^{\mathrm{norm}}} \scR^{\ur} \xrightarrow[\sim]{\theta^{\ur, -1}} \Lambda_{K}^{\ur} \xrightarrow[\sim]{\Tw_{r-1}} \Lambda_{K}^{\ur}.
\end{align*}
Here the twist $\Tw_{r-1}: \Lambda_{K} \rightarrow \Lambda_{K}$ denotes the $\calO$-linear isomorphism given by $\gamma \mapsto \lrangle{\epsilon(\gamma)}^{r-1} \gamma$ for $\gamma \in \Gamma_{\cyc}$ and by the identity on $\Gamma_{\ac}$. 

\begin{remark}
The purpose of this additional twist is to guarantee that the projection of $\calL_{p}^{\Gr}(f/K)$ to the anticyclotomic line is the anticyclotomic $p$-adic $L$-function in Theorem \ref{thm:BDP}.
\end{remark}

We then define two $p$-adic $L$-functions of $f$ over $K$ by a slight abuse of notation:
\begin{align*}
\calL_{p}^{\PR}(f/K)^{\eta} &:= \Cole(f/K)(\loc_{\frp}(\calZ^{\eta}(f/K))), \\ 
\calL_{p}^{\Gr}(f/K)^{\eta} &:= \Log(f/K)(\loc_{\bar{\frp}}(\calZ^{\eta}(f/K))).
\end{align*}
We call them the \emph{standard (Perrin--Riou) $p$-adic $L$-function} and the \emph{Greenberg $p$-adic $L$-function} of $f$ over $K$, respectively. Their interpolation properties follow from Theorem \ref{ERLI-thm} and Theorem \ref{thm:ERLIIint} respectively. We emphasize that the explicit reciprocity laws serve as the \emph{definitions} of these two-variable $p$-adic $L$-functions for $f$ over $K$. The following theorem then follows from Propositions \ref{prop:cyclotomic_descent} and \ref{prop:anticyclotomic_descent} (the involution $\jmath_{\ac}$ cancels the same involution appearing in the definition of $\theta$ in Section \ref{sec:more-iwasawa-algebra}).

\begin{theorem}[{Theorem \ref{main:C}}] \label{thm:descent_theorem}
With the notation above, the following statements hold.
\begin{enumerate}[label = \rm (\arabic*)]
        \item $(\calL^{\PR}_p(f/K)^{\eta} \pmod{I_{\ac}}) = (\calL_{p}(f/\QQ)^{\eta} \cdot \calL_{p}(f^{K}/\QQ)^{\eta})$, where $I_{\ac} := \ker(\Lambda_{K} \twoheadrightarrow \Lambda_{\cyc})$, and
        \item assuming \eqref{eq:gen-Heegner}, we have 
        \[
        (\calL^{\Gr}_p(f/K)^{\omega^{1-r}} \pmod{I_{\cyc}^{\ur}}) = (\calL_{p}^{\BDP}(f/K)),
        \]
        where $I_{\cyc}^{\ur} := \ker(\Lambda_{K}^{\ur} \twoheadrightarrow \Lambda_{\ac}^{\ur})$.
        \end{enumerate}
    Here $\calL_p(f/\QQ) \in \Lambda$ is the cyclotomic $p$-adic $L$-function of $f$ in Theorem \ref{thm:cyclotomic_p-adic_L-function}, and $\calL_{p}^{\BDP}(f/K) \in \Lambda_{\ac}^{\ur}$ is the Bertolini--Darmon--Prasanna $p$-adic $L$-function for $f$ over $K$ in Theorem \ref{thm:BDP}.
\end{theorem}

\begin{warning}[Convention on the tame character]
    From this point onward, we fix the tame character $\eta = \omega^{1-r}$ and omit it from the notation. This choice ensures the validity of the anticyclotomic descent theorem (Theorem \ref{thm:descent_theorem}(2)). The cyclotomic Iwasawa main conjecture for the fixed modular form $f$, stated in Theorem~\ref{main:A}, is intrinsic to $f$ and is independent of this auxiliary choice. Only the comparison below with the anticyclotomic Iwasawa theory of $f$ uses the branch $\eta=\omega^{1-r}$.
\end{warning}

\subsection{Some properties of Coleman and logarithm maps}
To prepare for the Poitou--Tate duality arguments in Section \ref{sec:PT}, we prove some additional properties of the Coleman map $\Cole(f/K)$ and the logarithm map $\Log(f/K)$. These are the properties implicitly used in the proof of \cite[Section 9.3.2]{BSTW}.

We start with the Coleman map.
\begin{proposition} \label{prop:pseudo_coleman}
    The Coleman map $\Cole(f/K): \rmH^{1}(K_{\frp}, \Fil^{-}\bfT_{K}) \rightarrow \Lambda_{K}$ is injective, with pseudonull cokernel as a $\Lambda_{K}$-module.
\end{proposition}

\begin{proof}
    It suffices to prove the result for the Coleman map $\Cole_{\frp}$. Its injectivity follows from its construction. 

    We next prove that its cokernel is a pseudonull $\Lambda_{K}$-module. By the proof of \cite[Proposition 5.25 (1)]{BSTW}, which carries over to our setting using the cyclotomic descent theorem (Theorem \ref{thm:descent_theorem}(1)), we have
\begin{equation} \label{eq:coleman_congruence}
\Cole_{\frp} \equiv \Cole_{\eta_{\omega}, \frp} \pmod{I_{\ac}}
\end{equation}
where $\Cole_{\eta_{\omega},\frp}$ is the Coleman map for the Beilinson--Kato zeta element recalled in \cite[Section 2.4.2]{BSTW}. Under condition \eqref{eq:irred} and the assumption that the differential $\omega$ is "good", the image of $\Cole_{\eta_{\omega}, \frp}$ has finite index in $\Lambda$ by \cite[Proposition 17.11]{Kato}; equivalently, the cokernel of $\Cole_{\eta_{\omega}, \frp}$ is a pseudonull $\Lambda_{\cyc}$-module. It then follows from \eqref{eq:coleman_congruence} that the cokernel of $\Cole(f/K)$ is a pseudonull $\Lambda_{K}$-module by \cite[Theorem 5.2]{QX}.
\end{proof}

\begin{proposition} \label{prop:pseudo_log}
    Assume that \eqref{eq:irred} and \eqref{eq:gen-Heegner} hold. The logarithm map $\Log(f/K): \rmH^{1}(K_{\barfrp}, \Fil^{+}\bfT_{K}) \rightarrow \Lambda_{K}^{\ur}$ is injective, with cokernel being a pseudonull $\Lambda_{K}^{\ur}$-module.
\end{proposition}

We now prove the proposition. The proof of Proposition \ref{prop:pseudo_log} follows the same general strategy as that of Proposition \ref{prop:pseudo_coleman}, although the details are more involved. Since $\theta^{\ur,-1}$ and $\Tw_{r-1}$ are isomorphisms, it suffices to prove the result for the normalized map $\Log_{\barfrp}^{\mathrm{norm}}$. It is proved in \cite[Lemma 5.29]{BSTW} that
\begin{equation}\label{eq:Log_congruence}
   u_{K} \cdot \Log_{\barfrp}^{\mathrm{norm}} \equiv \calL_{f}^{\PR} \pmod{I_{\cyc}^{\ur}} 
\end{equation}
where $u_{K} \in (\Lambda_{\ac}^{\ur})^{\times}$ is a unit depending only on $K$, and $\calL_{f}^{\PR}$ is the Perrin--Riou regulator map formulated in \cite[Section 5.5.3]{BSTW}. The regulator map $\calL_{f}^{\PR}$ is less explicit than the Coleman map $\Cole_{\eta_{\omega}, \frp}$ in the cyclotomic setting, so we first study it in more detail.

The Perrin--Riou regulator map $\calL^{\PR}$ in \cite[Section 5.5.3]{BSTW} is constructed as the composite
\begin{multline*}
\calL^{\PR}: \rmH^{1}(K_{\barfrp}, \Fil^{+}\bfT_{K}) \xrightarrow{\Tw} \rmH^{1}(K_{\barfrp}, \Fil^{+}T_{f} \hotimes_{\calO} \Lambda_{\ac}^{\iota}) \\ \xhookrightarrow{\res} \rmH^{1}(K_{\barfrp}^{\ur}, \Fil^{+}T_{f} \hotimes_{\calO} \Lambda_{\ac}^{\iota}) \xrightarrow{\calL_{T_{f}^{+}}} \DD_{\cris}(T_f^{+}) \hotimes W(\barFF_p) \hotimes \Lambda_{\ac},
\end{multline*}
where the first map removes the Tate twist $\ZZ_p(k-1)$ from $T$. For the third map, we use the fact that, as a $G_{K_{\barfrp}}^{\ur}$-module, $\Lambda_{\ac}$ is naturally isomorphic to the cyclotomic Iwasawa algebra $\Lambda$. For $f$, set
\[
\calL_f^{\PR} := ([\omega_f, -] \otimes \id{} \otimes \id{}) \circ \calL^{\PR}: \rmH^{1}(K_{\barfrp}, \Fil^{+}\bfT_{K}) \rightarrow \Lambda_{\ac}.
\]

We use the following theorem of Perrin--Riou \cite[Théorème 5]{PR2021} and Lei--Loeffler--Zerbes \cite[Theorem 4.16]{LLZ_Coleman} describing the image of $\calL_{T_{f}^{+}}$. We write $\scH(\Gamma)$ for the algebra of $\QQ_p$-valued distributions on $\Gamma$.
\begin{theorem}[Perrin-Riou, Lei--Loeffler--Zerbes] \label{thm:Perrin-Riou-image}
    Let $V$ be a crystalline $G_{\QQ_p}$-representation with nonnegative Hodge--Tate weights $0\le r_1\le\cdots\le r_d$. Define for $i \in \ZZ$
\[
\ell_i= \frac{\log(1+X)}{\log_p(\chi(\gamma))}-i
\]
and define $\lambda_k=\ell_0\ell_1\cdots\ell_{k-1}$ for $k \geq 1$ with $\lambda_0=1$. Assume that
\begin{equation} \label{eq:crystalline_condition} \tag{cris}
\text{no eigenvalue of crystalline Frobenius $\varphi$ acting on $V$ lies in $p^{\ZZ}$},     
\end{equation}
elementary divisors of the $\scH(\Gamma)$-module
\[
\frac{ \scH(\Gamma)\otimes_{\QQ_p}\DD_{\cris}(V)}{ \scH(\Gamma)\otimes_{\Lambda_{\QQ_p}(\Gamma)} \Im(\calL_V)
}
\]
are $[\lambda_{r_1};\ldots;\lambda_{r_d}]$.
\end{theorem}

To apply Theorem \ref{thm:Perrin-Riou-image} to $V = V_{f}^{+}$ with lattice $T_{f}^{+}$, we first verify assumption \eqref{eq:crystalline_condition}. The crystalline Frobenius $\varphi$ acts on $V_{f}^{+}$ by the $p$-adic unit $\alpha \in \ZZ_p^{\times}$. The only possibility for $\alpha \in p^{\ZZ}$ is $\alpha = 1$. This would imply that $\beta = p^{k-1}$ because $f$ has \emph{trivial nebentype}. Therefore $a_p(f) = 1+p^{k-1}$, contradicting Deligne's purity bound $\abs{a_p(f)} \leq 2p^{(k-1)/2}$ for an odd prime $p \geq 3$. \footnote{For modular forms $f$ with nontrivial nebentypus, this subtlety occurs only when $k=2$ and $p=3, 5$.} Thus $\alpha \not\in p^{\ZZ}$, and \eqref{eq:crystalline_condition} is satisfied.

The representation $V_{f}^{+}$ is one-dimensional with the unique Hodge--Tate weight $r_1 = 0$, and hence in Theorem \ref{thm:Perrin-Riou-image}, $\lambda_{r_1} = \lambda_{0} = 1$. It then follows that
\[
\scH(\Gamma) \otimes_{\Lambda_{\QQ_p}(\Gamma)} \Im(\calL_{V_f^{+}}) = \scH(\Gamma).
\]
Consequently,
\[
\scH(\Gamma) \otimes_{\Lambda_{\QQ_p}(\Gamma)} \Im(\calL_{f}^{\PR}) = \scH(\Gamma).
\]
Since a distinguished polynomial in the one-variable Iwasawa algebra $\calO\lrbracket{\Gamma}$ remains non-invertible in $\scH(\Gamma)$, it follows that $(\Im(\calL_{f}^{\PR}))_{\wp} = (\Lambda_{\ac}^{\ur})_{\wp}$ for every height-one prime ideal of $\Lambda_{\ac}^{\ur}$ generated by a distinguished polynomial. It remains to consider the case $\wp = (\varpi)$.

For this prime, the equality $(\Im(\calL_{f}^{\PR}))_{\wp} = (\Lambda_{\ac}^{\ur})_{\wp}$ for $\wp = (\varpi)$ follows directly from the anticyclotomic descent theorem (Theorem \ref{thm:descent_theorem}(2)) and Theorem \ref{thm:mu-vanishing} on the vanishing of the $\mu$-invariant of $\calL_{p}^{\BDP}(f/K)$. More precisely, by Theorem \ref{thm:descent_theorem}(2) and \eqref{eq:Log_congruence}, we obtain
\[
\calL_{f}^{\PR}(\calZ_{\ac}(f/K)) \equiv u_{K} \calL_p^{\BDP}(f/K) \pmod{I_{\cyc}}
\]
with $u_{K} \in \Lambda_{\ac}^{\times}$ a unit depending only on $K$. Theorem \ref{thm:mu-vanishing} implies that the right-hand side has $\mu$-invariant zero, or equivalently, is not divisible by $\varpi$. Thus $\varpi$ does not divide the image $\Im(\calL_{f}^{\PR})$, i.e. $(\Im(\calL_{f}^{\PR}))_{(\varpi)} = (\Lambda_{\ac}^{\ur})_{(\varpi)}$. 

Combining the two cases, we conclude that the cokernel of $\calL_{f}^{\PR}$ is a pseudonull $\Lambda_{\ac}^{\ur}$-module. Hence, by \eqref{eq:Log_congruence} and \cite[Theorem 5.2]{QX}, the cokernel of $\Log_{\barfrp}^{\mathrm{norm}}$ is also a pseudonull $\Lambda_{K}^{\ur}$-module. This proves Proposition \ref{prop:pseudo_log}. \hfill \qed

\subsection{Equivalence of Iwasawa main conjectures}
\label{sec:PT}

This section parallels \cite[Section 9.3.2]{BSTW}, with additional details.

\subsubsection{Poitou--Tate duality}
Fix a coefficient ring $R=\Lambda_\dagger$ with $\dagger \in \{K, \cyc, \ac\}$. We also use $\dagger \in \{\quad, \cyc, \ac\}$ as the corresponding subscript for zeta elements and $p$-adic $L$-functions. We write
\[
\scP_{\frp}(R) := \dfrac{\rmH^{1}(K_{\frp}, \bfT_{R})}{\rmH^{1}_{\ord}(K_{\frp}, \bfT_{R})}, \quad \scP_{\bar{\frp}}(R) := \rmH^{1}_{\ord}(K_{\bar{\frp}}, \bfT_{R}).
\]
Applying Poitou--Tate duality, we have the following exact sequences.

\begin{proposition}
\label{prop:PT-sequences}
There are exact sequences
\begin{align}
0\rightarrow&\rmH^1_{\ordi,\ordi}(K,\bfT_{R})
 \rightarrow \rmH^{1}_{\rel,\ord}(K, \bfT_{R})
 \xrightarrow{\loc_\frp}\scP_\frp(R)
 \rightarrow X_{\ordi, \ordi}
 \rightarrow X_{\rel,\ordi}\rightarrow0,
 \label{eq:PT1}\\
0\rightarrow&\rmH^1_{\rel,\str}(K,\bfT_{R})
 \rightarrow \rmH^{1}_{\rel,\ord}(K, \bfT_{R})
 \xrightarrow{\loc_\barfrp}\scP_\barfrp(R)
 \rightarrow X_{\rel, \str}
 \rightarrow X_{\rel,\ordi}\rightarrow0.
 \label{eq:PT2}
\end{align}
Here $\loc_\frp$ and $\loc_{\bar{\frp}}$ are the localization maps at $\frp$ and $\bar{\frp}$, respectively, and we write $X_{?} := X_{?}(K, \bfA_{R})$ to save space.
\end{proposition}
\begin{proof}
We recall the relevant part of Poitou--Tate duality. If $\calF\subseteq\calG$ are Selmer structures on $\bfT$ with identical conditions outside a finite set, global duality gives
\begin{multline}\label{eq:PT-general}
 0\rightarrow H^1_{\calF}(K,\bfT_{R})\rightarrow H^1_{\calG}(K,\bfT_{R}) \rightarrow \bigoplus_w\frac{H^1_{\calG}(K_w,\bfT_{R})}{H^1_{\calF}(K_w,\bfT_{R})} \rightarrow X_{\calF}(K, \bfA_{R}) \rightarrow X_{\calG}(K, \bfA_{R}) \rightarrow 0.
\end{multline}

Take $(\calF,\calG)=(\calF_{\ordi,\ordi},\calF_{\rel,\ordi})$.  The conditions differ only at $\frp$, and the local quotient is $\scP_{\frp}(R)$. This gives \eqref{eq:PT1}. Next take $(\calF,\calG)=(\calF_{\rel,\str},\calF_{\rel,\ordi})$.  The conditions differ only at $\barfrp$, and the local quotient is $\scP_{\barfrp}(R)$. This gives \eqref{eq:PT2}.
\end{proof}

\subsubsection{Nonvanishing theorems for zeta elements and properties of Iwasawa cohomology groups}
We write $Q_{\dagger}(f/K) := \rmH^{1}_{\rel, \ordi}(K, \bfT_{R})/R \cdot \calZ_{\dagger}(f/K)$ for $R = \Lambda_{K}, \Lambda_{\cyc}$ or $\Lambda_{\ac}$ with corresponding $\dagger \in \{\quad, \cyc, \ac\}$.

\begin{lemma}\label{lem:remove-kernels}
Suppose that the following two conditions hold:
\begin{align} \tag{rank-one}
    &\text{The Selmer group $\rmH^{1}_{\rel,\ord}(K, \bfT_{R})$ is torsion-free of rank one.} \label{eq:r1} \\
\label{eq:nv} \tag{nv}
&\text{Both $\loc_\frp(\calZ_{\dagger}(f/K))$ and $\loc_\barfrp(\calZ_{\dagger}(f/K))$ are non-torsion.}   
\end{align}
Then $\calZ_{\dagger}(f/K)$ is nonzero and $Q_{\dagger}(f/K)$ is a torsion $\Lambda$-module. Moreover, $\rmH^1_{\ordi,\ordi}(K,\bfT_{R})=0$ and $\rmH^1_{\rel,\str}(K,\bfT_{R})=0$.
\end{lemma}
\begin{proof}
The non-torsion of either localization implies $\calZ_{\dagger}(f/K) \neq 0$.  Since $\rmH^{1}_{\rel,\ord}(K, \bfT_{R})$ is torsion-free of rank one, the quotient $Q_{\dagger}(f/K) = \rmH^{1}_{\rel,\ord}(K, \bfT_{R})/R \cdot \calZ_{\dagger}(f/K)$ has rank zero and is therefore torsion.  The map $\loc_\frp: \rmH^{1}_{\rel,\ord}(K, \bfT_{R}) \to\scP_\frp(R)$ has nonzero image after tensoring with $\Frac(R)$, hence its kernel has rank zero. That kernel is a submodule of the torsion-free module $\rmH^{1}_{\rel,\ord}(K, \bfT_{R})$, so it is zero. Exactness of \eqref{eq:PT1} identifies the kernel with $H^1_{\ordi,\ordi}(K,\bfT_{R})$, which is then zero. The same argument applied to $\loc_{\barfrp}$ and \eqref{eq:PT2}, proves the second vanishing.
\end{proof}

We verify assumptions \eqref{eq:r1} and \eqref{eq:nv} when $R = \Lambda_{K}$. The case $R = \Lambda_{\cyc}$ follows from Proposition \ref{prop:rank}. The case $R = \Lambda_{\ac}$ is not used here, so we omit its discussion; see \cite[Remark 9.19 (ii)]{BSTW}. We therefore treat only $R = \Lambda_{K}$.

\begin{proposition} \label{prop:rank-one-nonvanishing}
Assume that $f$ is a $p$-ordinary newform of even weight $k = 2r \geq 2$ with $p \nmid N$, and that \eqref{eq:irred} holds. Then the module $\rmH^{1}_{\rel, \ord}(K, \bfT_{K})$ is
torsion-free of rank one over $\Lambda_{K}$. Under the assumption \eqref{eq:gen-Heegner}, the condition \eqref{eq:nv} holds.
\end{proposition}

\begin{proof}
We divide the proof into three steps.

\underline{Step 1: We show that $\rmH^{1}(G_{K,S}, \bfT_{K})$ is a torsion-free $\Lambda_{K}$-module.} The proof is similar to that of \cite[Section 13.8]{Kato}, but we include the details. To show that $\rmH^{1}(G_{K,S}, \bfT_{K})$ is a torsion-free $\Lambda_{K}$-module, it is enough to prove that multiplication by every irreducible element $f \in \Lambda_{K}$ is injective on $\rmH^{1}(G_{K,S}, \bfT_{K})$. Consider the short exact sequence
\[
0 \rightarrow \bfT_{K} \xrightarrow{\times f} \bfT_{K} \rightarrow \bfT_{K}/f \bfT_{K} \rightarrow 0.
\]
Taking the global cohomology yields
\[
0 \rightarrow \rmH^{0}(G_{K,S}, \bfT_{K}/f \bfT_{K}) \rightarrow \rmH^{1}(G_{K,S}, \bfT_{K}) \xrightarrow{\times f} \rmH^{1}(G_{K,S}, \bfT_{K}).
\]
Thus it is enough to show that $\rmH^{0}(G_{K,S}, \bfT_{K}/f \bfT_{K}) = 0$.

Suppose, to the contrary, that there exists a nonzero element $m \in \rmH^{0}(G_{K,S}, \bfT_{K}/f \bfT_{K})$. The module $\bfT_{K}/f \bfT_{K} \simeq T \hotimes_{\calO} (\Lambda_{K}/f)^{\iota}$ is free of rank two over the integral domain $\Lambda_{K}/f$. Hence the image 
\[
m \otimes 1 \in \bfT_{K}/f \bfT_{K} \otimes_{\Lambda_{K}/f} \Frac(\Lambda_{K}/f)
\] is nonzero and fixed by $G_{K,S}$. It therefore suffices to show that $\bfT_{K}/f \bfT_{K} \otimes_{\Lambda_{K}/f} \Frac(\Lambda_{K}/f)$ has no nonzero $G_{K,S}$-invariant elements. We observe that
\begin{equation} \label{eq:torsion-free-quotient}
\bfT_{K}/f \bfT_{K} \otimes_{\Lambda_{K}/f} \Frac(\Lambda_{K}/f) \simeq \begin{cases}
    V \otimes_{E} \Frac(\Lambda_{K}/f)(\Psi_{f}^{-1}), \quad& (f) \neq (\varpi), \\ 
    \barT \otimes_{\FF} \Frac(\Lambda_{K}/f)(\overline{\Psi}^{-1}), \quad& (f) = (\varpi),
\end{cases}
\end{equation}
where $\Psi_{f}: G_{K} \rightarrow \Lambda_{K}^{\times} \rightarrow (\Lambda_{K}/f)^{\times} \rightarrow \Frac(\Lambda_{K}/f)^{\times}$ is the quotient of the tautological character, and $\overline{\Psi}$ is the reduction of the tautological character modulo $\varpi$. In either case, any nonzero $G_{K,S}$-invariant element on the left-hand side of \eqref{eq:torsion-free-quotient} would generate a one-dimensional $\Gal_{K,S}$-stable subspace on the right-hand side, contradicting \eqref{eq:irred}. This proves the result. 

Consequently, the submodule $\rmH^{1}_{\rel, \ord}(K, \bfT_{K})$ of $\rmH^{1}(G_{K,S}, \bfT_{K})$ is also a torsion-free $\Lambda_{K}$-module.

\underline{Step 2: We compute that $\rank_{\Lambda} \rmH^{1}_{\rel, \ord}(K, \bfT_{K}) = 1$.} 

Recall $I_{\ac} := (T_{\ac} := \gamma_{\ac} - 1)$. To compute its $\Lambda_{K}$-rank, we reduce it to the $\Lambda_{\cyc}$-rank of $\rmH^{1}_{\rel, \ord}(K, \bfT_{\cyc})$. Note that there is a short exact sequence
\[
0 \rightarrow \bfT_{K} \xrightarrow{\times T_{\ac}} \bfT_{K} \xrightarrow{\mathrm{sp}} \bfT_{\cyc} \rightarrow 0.
\]
We take the global Galois cohomology $\rmH^{1}(G_{K,S}, -)$ to obtain
\begin{equation} \label{eq:long_exact_H1}
\rmH^{1}(G_{K,S}, \bfT_{K}) \xrightarrow{\times T_{\ac}} \rmH^{1}(G_{K,S}, \bfT_{K}) \xrightarrow{\mathrm{sp}} \rmH^{1}(G_{K,S}, \bfT_{\cyc}),
\end{equation}
which then induces an injection
\[
\mathrm{sp}: \dfrac{\rmH^{1}(G_{K,S}, \bfT_{K})}{T_{\ac} \cdot \rmH^{1}(G_{K,S}, \bfT_{K})} \hookrightarrow \rmH^{1}(G_{K,S}, \bfT_{\cyc}).
\]
We claim that this injection induces
\begin{equation} \label{eq:specliazation_control}
\mathrm{sp}: \dfrac{\rmH^{1}_{\rel, \ord}(K, \bfT_{K})}{T_{\ac} \cdot \rmH^{1}_{\rel, \ord}(K, \bfT_{K})} \hookrightarrow \rmH^{1}_{\rel, \ord}(K, \bfT_{\cyc})
\end{equation}

Assuming this claim, we continue the proof of Step 2. By \eqref{eq:specliazation_control},
\begin{align*}
\rank_{\Lambda_{K}} \rmH^{1}_{\rel, \ord}(K, \bfT_{K}) = \rank_{\Lambda_{\cyc}} \dfrac{\rmH^{1}_{\rel, \ord}(K, \bfT_{K})}{T_{\ac} \cdot \rmH^{1}_{\rel, \ord}(K, \bfT_{K})} \leq \rank_{\Lambda_{\cyc}} \rmH^{1}_{\rel, \ord}(K, \bfT_{\cyc}).
\end{align*}
We have seen in (1) that $\rank_{\Lambda_{\cyc}} \rmH^{1}_{\rel, \ord}(K, \bfT_{\cyc}) = 1$, and hence 
\[
\rank_{\Lambda_{K}} \rmH^{1}_{\rel, \ord}(K, \bfT_{K}) \leq 1 . 
\]
To prove equality, by Corollary \ref{coro:cyclo_nonvanishing}, we have $\calL_{p}^{\PR}(f/K) \neq 0$, and hence $\calZ(f/K)$ is nonzero. Since $\rmH^{1}_{\rel, \ord}(K, \bfT_{K})$ is a torsion-free $\Lambda_{K}$-module, it has positive rank, and therefore
\[
\rank_{\Lambda_{K}} \rmH^{1}_{\rel, \ord}(K, \bfT_{K}) = 1,
\]
as desired. 

\underline{Step 3: The nontorsionness of $\loc_{\frp}(\calZ(f/K))$ and $\loc_{\barfrp}(\calZ(f/K))$ follows.}

By Corollary \ref{coro:cyclo_nonvanishing}, the localization $\loc_{\frp}(\calZ(f/K))$ is nontorsion. The same holds for $\loc_{\barfrp}(\calZ(f/K))$ by combining Theorem \ref{thm:descent_theorem}(2) and Theorem \ref{thm:mu-vanishing}, under assumption \eqref{eq:gen-Heegner}.
\end{proof}

We now prove \eqref{eq:specliazation_control} as used above.

\begin{proof}[Proof of \eqref{eq:specliazation_control}]
For $w \in S$, we write 
\[
\rmH^{1}_{/\calF}(K_{w}, \bfT_{K}) := \dfrac{\rmH^{1}(K_{w}, \bfT_{K})}{\rmH^{1}_{\calF}(K_{w}, \bfT_{K})}
\]
for simplicity. It suffices to show that $\rmH^{1}_{/\calF}(K_{w}, \bfT_{K})[T_{\ac}] = 0$ for each $w \in S$. We consider the following three types of places in $S$ one-by-one:

\underline{(1) $w = \frp$.} In this case, $\rmH^{1}_{\calF}(K_{w}, \bfT_{K}) = \rmH^{1}_{\rel}(K_{w}, \bfT_{K}) = \rmH^{1}(K_{w}, \bfT_{K})$. Hence the quotient $\rmH^{1}_{/\calF}(K_{w}, \bfT_{K})$ is zero.

\underline{(2) $w = \barfrp$.} In this case, 
\[
\rmH^{1}_{/\calF}(K_{w}, \bfT_{K}) = \dfrac{\rmH^{1}(K_{w}, \bfT_{K})}{\rmH^{1}_{\ord}(K_{w}, \bfT_{K})} = \im(\rmH^{1}(K_{\barfrp}, \bfT_{K}) \rightarrow \rmH^{1}(K_{\barfrp}, \Fil^{-} \bfT_{K})) \hookrightarrow \rmH^{1}(K_{\barfrp}, \Fil^{-} \bfT_{K}).
\]
Thus it suffices to show that $\rmH^{1}(K_{\barfrp}, \Fil^{-} \bfT_{K})[T_{\ac}] = 0$. Consider the short exact sequence
\[
0 \rightarrow \Fil^{-} \bfT_{K} \xrightarrow{\times T_{\ac}} \Fil^{-} \bfT_{K} \xrightarrow{\sp} \Fil^{-} \bfT_{\cyc} \rightarrow 0.
\]
Taking local Galois cohomology gives
\[
\rmH^{0}(K_{\barfrp}, \Fil^{-} \bfT_{\cyc}) \rightarrow \rmH^{1}(K_{\barfrp}, \Fil^{-} \bfT_{K}) \xrightarrow{\times T_{\ac}} \rmH^{1}(K_{\barfrp}, \Fil^{-} \bfT_{K}).
\]
It remains to show that $\rmH^{0}(K_{\barfrp}, \Fil^{-} \bfT_{\cyc}) = 0$. Since $\Fil^{-} T$ is unramified, for every $\tau \in I_{\barfrp}$ the action of $\tau$ is trivial on $\Fil^{-} T$, while $\tau$ acts on $\Lambda_{\cyc}$ by $\Psi_{\cyc}(\tau)^{-1}$. Hence $(\Psi_{\cyc}(\tau) - 1) \rmH^{0}(K_{\barfrp}, \Fil^{-} \bfT_{\cyc}) = 0$. Because $K_{\infty}^{\cyc}/K$ is totally ramified at every prime above $p$, the map $\Psi_{\cyc}: I_{\barfrp} \twoheadrightarrow \Gamma_{\cyc}$ is surjective. Choose $\tau \in I_{\barfrp}$ whose image is $\gamma_{\cyc}$. Then $(\Psi_{\cyc}(\tau) - 1)$ is a nonzerodivisor in $\Lambda_{\cyc}$. Since $\rmH^{0}(K_{\barfrp}, \Fil^{-} \bfT_{\cyc}) \subseteq \Fil^{-} \bfT_{\cyc})$ is torsion-free over $\Lambda_{\cyc}$, it follows that $\rmH^{0}(K_{\barfrp}, \Fil^{-} \bfT_{\cyc}) = 0$, as desired.

\underline{(3) $v \in S$ outside $p$.} In this case, 
\[
\rmH^{1}_{/\calF}(K_{w}, \bfT_{K}) = \dfrac{\rmH^{1}(K_{w}, \bfT_{K})}{\rmH^{1}_{\ur}(K_{w}, \bfT_{K})} = \im(\rmH^{1}(K_w, \bfT_{K})\rightarrow \rmH^{1}(I_w, \bfT_{K})) \hookrightarrow \rmH^{1}(I_{w}, \bfT_{K}).
\]
Since $K_{\infty}/K$ is unramified outside $p$, the action of $I_{w}$ is trivial over $\Lambda_{K}^{\iota}$. This gives $\bfT_{K}|_{I_{w}} = T|_{I_{w}} \otimes_{\calO} \Lambda_{K}^{\iota}$, and hence
\[
\rmH^{1}(I_{w}, \bfT_{K}) \simeq \rmH^{1}(I_{w}, T) \otimes_{\calO} \Lambda_{K}^{\iota},
\]
whose right-hand side has no $T_{\ac}$-torsion, as desired.

We now prove the injectivity in \eqref{eq:specliazation_control}. Suppose that $c \in \rmH^{1}_{\rel, \ord}(K, \bfT_{K})$ satisfies $\mathrm{sp}(c) = 0$. It follows from \eqref{eq:long_exact_H1} that $c = T_{\ac} \cdot b$ for some $b \in \rmH^{1}(G_{K,S}, \bfT_{K})$. We claim that $b \in \rmH^{1}_{\rel, \ord}(K, \bfT_{K})$. Indeed, since $c \in \rmH^{1}_{\rel, \ord}(K, \bfT_{K})$, for every $w \in S$,
\[
\loc_{w}(c) = T_{\ac} \cdot \loc_{w}(b) \in \rmH^{1}_{\calF}(K_{w}, \bfT_{K}).
\]
In other words, $T_{\ac} \cdot \loc_{w}(b) = 0 \in \rmH^{1}_{/\calF}(K_{w}, \bfT_{R})$. We have shown above that $\rmH^{1}_{/\calF}(K_{w}, \bfT_{R})$ has no nonzero $T_{\ac}$-torsion, hence this forces $\loc_{w}(b) \in \rmH^{1}_{\calF}(K_{w}, \bfT_{K})$. This implies $b \in \rmH^{1}_{\rel, \ord}(K, \bfT_{K})$. We have therefore obtained that the map \eqref{eq:specliazation_control} is well-defined and 
\[
\ker(\mathrm{sp}: \rmH^{1}_{\rel, \ord}(K, \bfT_{K}) \rightarrow \rmH^{1}_{\rel, \ord}(K, \bfT_{\cyc})) \subseteq T_{\ac} \rmH^{1}_{\rel, \ord}(K, \bfT_{K}).
\]
The converse inclusion is clear from \eqref{eq:long_exact_H1}. The equation \eqref{eq:specliazation_control} then follows.
\end{proof}

\subsubsection{Equivalence of Iwasawa main conjectures}
Quotienting \eqref{eq:PT1} and \eqref{eq:PT2} by the submodule $R \cdot \calZ(f/K)$ and using Lemma~\ref{lem:remove-kernels} gives
\begin{align}
0\rightarrow Q(f/K)&\rightarrow \frac{\scP_\frp(\Lambda_{K})}{\Lambda_{K} \cdot \loc_\frp(\calZ(f/K))} \rightarrow X_{\ordi, \ordi}(K, \bfA_{K})\rightarrow
 X_{\rel,\ordi}(K, \bfA_{K})\rightarrow0,
 \label{eq:PTQ1}\\
0\rightarrow Q(f/K)&\rightarrow
 \frac{\scP_\barfrp(\Lambda_{K})}{ \Lambda_{K} \cdot \loc_{\barfrp}(\calZ(f/K))}
 \rightarrow X_{\rel, \str}(K, \bfA_{K})\rightarrow
 X_{\rel,\ordi}(K, \bfA_{K})\rightarrow0.
 \label{eq:PTQ2}
\end{align}

The explicit reciprocity laws of $\calZ(f/K)$ are used in the following way.
\begin{lemma} \label{lem:regulator-quotient}
Let $R$ be a noetherian normal domain, let $P$ be a torsion-free rank-one $R$-module, and let $\Phi:P\hookrightarrow R$ be an injection with pseudo-null cokernel.  If $z\in P$ and $\Phi(z)=L\ne0$, then the map induced by $\Phi$,
\[
P/Rz\rightarrow R/(L),
\]
is injective and has pseudo-null cokernel. In particular, $\chari_R(P/Rz)=(L)$. Applying this to $z=\calZ(f/K)$ gives the following pseudo-isomorphism of $\Lambda_{K}$-modules
\begin{equation}
 \frac{\scP_\frp(\Lambda_{K})}{\Lambda_{K} \cdot\loc_\frp(\calZ(f/K))} \xrightarrow{\approx} \frac{\Lambda_{K}}{\Lambda_{K} \cdot \calL_{p}^{\PR}(f/K)},
 \label{eq:local-quotient-one}   
\end{equation}
and the following pseudo-isomorphism of $\Lambda_{K}^{\ur}$-modules under assumption \eqref{eq:gen-Heegner}:
\begin{equation}
    \frac{\scP_\barfrp(\Lambda_{K})^{\ur}}{\Lambda_{K}^{\ur} \cdot \loc_\barfrp(\calZ(f/K))} \xrightarrow{\approx} \frac{\Lambda_{K}^{\ur}}{\Lambda_{K}^{\ur}\cdot \calL^{\Gr}_{p}(f/K)}.
 \label{eq:local-quotient-two}
\end{equation}
Here we use the symbol $\approx$ to mean that the morphism is a pseudo-isomorphism.
\end{lemma}
\begin{proof}
If $\Phi(x)\in(L)$, write $\Phi(x)=aL=\Phi(az)$ for some $a\in R$. Injectivity of $\Phi$ gives $x=az$, proving injectivity of the induced map.  Its cokernel is
\[
R/\bigl(\Phi(P)+(L)\bigr),
\]
which is a quotient of $R/\Phi(P)$ and is therefore pseudo-null. Applying this algebraic result to the zeta elements $\calZ_{\dagger}(f/K)$, it suffices to verify that the maps $\Cole(f/K)$ and $\Log(f/K)$ are injective with pseudo-null cokernel. This follows from Propositions \ref{prop:pseudo_coleman} and \ref{prop:pseudo_log} respectively.
\end{proof}

Apply the first local quotient identity \eqref{eq:local-quotient-one} to \eqref{eq:PTQ1}.  Apply $-\hotimes_{\Lambda_{K}} \Lambda_{K}^{\ur}$ to \eqref{eq:PTQ2} and then use \eqref{eq:local-quotient-two}.  Faithful flatness preserves exactness. We therefore obtain, going parallel with \cite[(9.13) and (9.12)]{BSTW} \footnote{In \cite[(9.13) and (9.12)]{BSTW}, the second terms are the quotient modules $\Lambda_{K}/(\calL_{p}^{\PR}(f/K))$ and $\Lambda_{K}^{\ur}/(\calL^{\Gr}_{p}(f/K))$ directly. We shall understand them as modules pseudo-isomorphic to them over $\Lambda_{K}$ or $\Lambda_{K}^{\ur}$ accordingly. The proofs in \textit{loc.cit.} still carry on.}, the following exact sequences:
\begin{align}
0\rightarrow Q(f/K)&\rightarrow[\Lambda_{K}/(\calL_{p,\dagger}^{\PR}(f/K))]\rightarrow X_{\ordi, \ordi}(K, \bfA_{K}) \rightarrow X_{\rel,\ordi}(K, \bfA_{K})\rightarrow0, \label{eq:four-standard}\\
0\rightarrow Q(f/K)^{\ur}&\rightarrow [\Lambda_{K}^{\ur}/(\calL^{\Gr}_{p}(f/K))] \rightarrow X_{\rel, \str}^{\ur}(K, \bfA_{K}) \rightarrow X_{\rel,\ordi}^{\ur}(K, \bfA_{K})\rightarrow0.\label{eq:four-greenberg}
\end{align}
Here, for a \emph{particular} finitely generated torsion $R$-module $M$, we write $M = [M_0]$ when $M_0$ is another such $R$-module pseudo-isomorphic to $M$.

Since $Q(f/K)$, $Q(f/K)^{\ur}$, and the second terms of \eqref{eq:four-standard} and \eqref{eq:four-greenberg} are torsion over $\Lambda_{K}$ or $\Lambda_{K}^{\ur}$, respectively, by Lemma \ref{lem:remove-kernels}, taking ranks gives
\[
\rank_{\Lambda_{K}} X_{\ordi,\ordi}(K, \bfA_{K})=\rank_{\Lambda_{K}} X_{\rel,\ordi}(K, \bfA_{K})
\]
and
\[
\rank_{\Lambda_{K}^{\ur}}X_{\rel, \str}^{\ur}(K, \bfA_{K}) =\rank_{\Lambda_{K}^{\ur}}X_{\rel,\ordi}^{\ur}(K, \bfA_{K}).
\]
Thus any one of these three dual Selmer modules is torsion if and only if the other two are torsion. We assume these equivalent conditions in what follows.

We fix a pair of ideals $\wp$ and $\twp$, where $\wp$ (resp. $\twp$) is any height-one prime ideal of $\Lambda_{K}$ (resp. $\Lambda_{K}^{\ur}$), such that $\twp \cap \Lambda_{K} = \wp$. We note that, since the ring extension $\Lambda_{K} \rightarrow \Lambda_{K}^{\ur}$ is faithfully flat and unramified in codimension one \footnote{By saying ``\emph{unramified in codimension one}'', we simply mean that the ring extension of discrete valuation rings $\Lambda_{K, \wp} \rightarrow (\Lambda_{K}^{\ur})_{\twp}$ is unramified, since $\calO \rightarrow \calO^{\ur}$ is unramified.}, for any finitely generated torsion $\Lambda_{K}$-module $M$, we have
\[
\length_{(\Lambda_{K}^{\ur})_{\twp}}\left(( M \otimes_{\Lambda_{K}} \Lambda_{K}^{\ur})_{\twp}\right) =  \length_{\Lambda_{K, \wp}} (M_{\wp}).
\]
In particular for any nonzero $\calL \in \Lambda_{K}$, we have
\[
\length_{(\Lambda_{K}^{\ur})_{\twp}} \left((\Lambda_{K}^{\ur})_{\twp}/\calL \cdot (\Lambda_{K}^{\ur})_{\twp} \right) =  \length_{\Lambda_{K, \wp}} (\Lambda_{K,\wp}/\calL \cdot \Lambda_{K,\wp}).
\]

Under these equivalent conditions, fixing the pair $(\wp, \twp)$ chosen above, and taking $\length_{\Lambda_{K,\wp}}$ (resp. $\length_{(\Lambda_{K}^{\ur})_{\twp}}$) along the exact sequence \eqref{eq:four-standard} (resp. \eqref{eq:four-greenberg}), we have
\begin{align*}
& \length_{\Lambda_{K,\wp}}(X_{\ordi,\ordi} (K, \bfA_{K})_{\wp}) - \length_{\Lambda_{K,\wp}}(\Lambda_{K,\wp}/\calL_p^{\PR}(f/K)) \\ \xlongequal{\text{\eqref{eq:four-standard}}} & \length_{\Lambda_{K,\wp}}(X_{\rel,\ordi}(K, \bfA_{K})_{\wp})- \length_{\Lambda_{K,\wp}}(Q(f/K)_{\wp}) \\
\xlongequal{\text{base change}} & \length_{(\Lambda_{K}^{\ur})_{\twp}}(X_{\rel,\ordi}^{\ur}(K, \bfA_{K})_{\twp})- \length_{(\Lambda_{K}^{\ur})_{\twp}}(Q(f/K)^{\ur}_{\twp}) 
\\ \xlongequal{\text{\eqref{eq:four-greenberg}}} & \length_{(\Lambda_{K}^{\ur})_{\twp}}(X_{\rel, \str}^{\ur}(K, \bfA_{K})_{\twp})-\length_{(\Lambda_{K}^{\ur})_{\twp}}((\Lambda_{K}^{\ur})_{\twp}/\calL_p^{\Gr}(f/K)).
\end{align*}
This proves the following main result.

\begin{theorem}[Equivalency of Iwasawa main conjectures] \label{thm:equiv_IMC}
Assume that \eqref{eq:irred} and \eqref{eq:gen-Heegner} hold. Then $X_{\ord, \ord}(K, \bfA_{K})$ is a finitely generated torsion $\Lambda_{K}$-module if and only if $X_{\rel, \ord}^{\ur}(K, \bfA_{K})$ is a finitely generated torsion $\Lambda_{K}^{\ur}$-module. Moreover, under these two equivalent conditions, for any pair of height-one prime ideals $\wp$ and $\twp$ of $\Lambda_{K}$ and $\Lambda_{K}^{\ur}$, respectively, such that $\twp \cap \Lambda_{K} = \wp$, we have
\begin{multline*}
  \length_{\Lambda_{K,\wp}}( X_{\ordi,\ordi}(K, \bfA_{K})_{\wp}) - \length_{\Lambda_{K,\wp}}(\Lambda_{K, \wp}/\calL_p^{\PR}(f/K)) \\ 
  = \length_{(\Lambda_{K}^{\ur})_{\twp}}( X_{\rel, \str}^{\ur}(K, \bfA_{K})_{{\twp}})- \length_{(\Lambda_{K}^{\ur})_{\twp}}((\Lambda_{K}^{\ur})_{\twp}/\calL_p^{\Gr}(f/K))  
\end{multline*}
In particular, for Perrin-Riou's Iwasawa main conjecture
\begin{equation} \tag{PRIMC} \label{eq:PRMC}
    \chari_{\Lambda_{K}} X_{\ordi,\ordi}(K, \bfA_{K}) = (\calL_p^{\PR}(f/K)) \quad \text{ in $\Lambda_{K}$},
\end{equation}
and Greenberg's Iwasawa main conjecture
\begin{equation} \tag{GIMC} \label{eq:GIMC}
    \chari_{\Lambda_{K}^{\ur}} X_{\rel, \str}^{\ur}(K, \bfA_{K}) = (\calL_p^{\Gr}(f/K)) \quad \text{ in $\Lambda_{K}^{\ur}$},
\end{equation}
we have the following equivalences:
\begin{enumerate}[label = \rm (\arabic*)]
    \item The one-sided divisibility in \eqref{eq:PRMC} is equivalent to the corresponding divisibility in \eqref{eq:GIMC}, as ideals of $\Lambda_{K}$ and $\Lambda_{K}^{\ur}$, respectively.
    \item The one-sided divisibility in \eqref{eq:PRMC} is equivalent to the corresponding divisibility in \eqref{eq:GIMC}, as ideals of $\Lambda_{K}[1/\varpi]$ and $\Lambda_{K}^{\ur}[1/\varpi]$, respectively.
    \item The validity of \eqref{eq:PRMC} is equivalent to the validity of \eqref{eq:GIMC}, as equalities of ideals in $\Lambda_{K}$ and $\Lambda_{K}^{\ur}$, respectively.
\end{enumerate}
\end{theorem}

From now on, we shall write
\[
X_{\ordi}(X/R) := X_{\ordi,\ordi}(K, \bfA_{R}), \quad X_{\Gr}(X/R) := X_{\rel, \str}(K, \bfA_{R}),
\]
for $R = \Lambda_{K}, \Lambda_{\cyc}$ and $\Lambda_{\ac}$. We call them the \emph{ordinary Selmer group} and the \emph{Greenberg Selmer group}, respectively.

%% file: 05_Iwasawa_main_conjectures.tex
\section{Applications to Iwasawa main conjectures}

\subsection{Burungale--Castella--Skinner's base change arguments}
\label{sec:base-change}

We now combine the integral equivalence with Wan's Hilbert modular form divisibility.  The argument in this section is the higher-weight form of \cite[Sections~3--5]{BCS}.  Its central result is Theorem~\ref{thm:BCS-521-higher}, which is the promised analogue of \cite[Proposition~5.2.1]{BCS}.

\subsubsection{The quartic CM field and auxiliary hypotheses}

Let $F/\QQ$ be a real quadratic field, let $K/\QQ$ be an imaginary quadratic field, and put $M=FK$, so that $M/F$ is a CM field extension. Denote by $\eta_F$, $\eta_K$, and $\eta_{FK}$ the three quadratic Dirichlet characters. For a form $g$, write $g^F=g\otimes\eta_F$, and similarly for the other twists.  The base change of $f$ to $F$ is denoted $f_F$.

For later reference, we collect the following seven conditions on $(K,F)$:
\begin{enumerate}[label=\rm {(BC\arabic*)}]
\item\label{BC:K} $D_K$ is odd, $D_K\ne3$, every prime dividing $N$ splits in $K$, and $p=\frp\barfrp$ splits in $K$ (these are conditions \cite[(disc), (Heeg), (spl)]{BCS});
\item\label{BC:Fp} $p$ is inert in $F$ (this is the condition \cite[(i) in Proposition 5.2.1]{BCS});
\item\label{BC:DF} $D_F$ is odd and every prime dividing $D_F$ splits in $K$ (this is the condition \cite[(ii) in Proposition 5.2.1]{BCS});
\item\label{BC:level} for every $\ell\mid N$, the prime $\ell$ is inert in $F$ if $\ell\equiv-1\pmod p$, and is split in $F$ otherwise (this is the condition \cite[(iii) in Proposition 5.2.1]{BCS});
\item\label{BC:irrM} $\overline T|_{G_M}$ is absolutely irreducible (this is the condition \cite[(iv) in Proposition 5.2.1]{BCS});
\item\label{BC:irrFzeta} $\overline T|_{G_{F(\mu_p)}}$ is absolutely irreducible (this is the condition \cite[(v) in Proposition 5.2.1]{BCS});
\item\label{BC:p5} if $p=5$, then $F\ne\QQ(\mu_5)^+$ (this is the condition \cite[(vi) in Proposition 5.2.1]{BCS}).
\end{enumerate}
Note that \ref{BC:irrM} implies that $\barT|_{G_K}$ is absolutely irreducible.

These conditions are chosen so that every hypothesis in Wan's theorem \cite[Theorem 3]{Wan2015} can be checked directly. This is verified in the proof of \cite[Proposition 5.2.1]{BCS}. Indeed:
\begin{itemize}
\item $p$ is unramified in $F$, and the unique prime of $F$ over $p$ splits in $M/F$, by \ref{BC:K} and \ref{BC:Fp};
\item $(pN,D_F)=1$ and $(N\calO_F,D_{M/F})=1$, by \ref{BC:DF} and \ref{BC:level};
\item Wan's condition $(\triangle)$ holds.  A prime dividing $D_F$ splits in $M/F$ by \ref{BC:DF}, while a prime dividing $D_K$ ramifies in $M/F$, so $M$ is not contained in the narrow Hilbert class field of $F$;
\item \ref{BC:DF} implies that every prime of $F$ dividing the level of $f_F$ splits in $M$. Thus the inert part $\frn^-$ of the factorization $N \calO_{F} = \frn^{+} \frn^{-}$ is $1$; the conditions on $\frn^-$ are therefore vacuous;
\item the Gorenstein criterion of Fujiwara quoted as \cite[Theorem~8]{Wan2015} applies:
\begin{itemize}
    \item The residual irreducibility of $\barT|_{G_{F(\mu_p)}}$ is \ref{BC:irrFzeta}, and its exceptional $p=5$ case is excluded by \ref{BC:p5}.
    \item Since $f$ is $p$-ordinary, $p \nmid N$, and has trivial nebentype, the two characters occurring in the semisimplification of $\barT|_{I_{p}}$ have ratio $\barchi_{\cyc}^{k-1}$. As $k$ is even and $p$ is odd, $p - 1 \nmid k-1$, and hence this ratio is nontrivial. Moreover, $p$ is unramified in $F$, so $I_{F_v} = I_{\QQ_p}$ for every $v \mid p$. Thus $\barT|_{G_{F_v}}$ is distinguished for every $v \mid p$. This verifies the condition (dist) in \cite{Wan2015};
    \item The assumption on the validity of Ihara lemma occurs only when $[F:\QQ]$ is odd. Here $[F:\QQ] = 2$, so it is vacuous;
    \item The existence of a minimal modular lifting of $\barT$ follows from the proof of \cite[Theorem~103]{Wan2015}. Indeed, by the choice of $F$, the base change to $F$ of a minimal modular lifting of $\barT$ gives a minimal modular lifting of $\barT_{f_{F}}$.
    \item Its final local condition, namely that ``for any finite place $v$ of $F$, if $\barT|_{G_{F_v}}$ is absolutely irreducible and $\barT|_{I_{v}}$ is absolutely irreducible, then $q_v \not\equiv -1 \pmod{p}$'', follows from \ref{BC:level}.
\end{itemize}
\end{itemize}

\subsubsection{Wan's Iwasawa main conjecture for Hilbert Hida families}

For $M = FK$, let $\Gamma_{M}^{\ac}$ (resp. $\Gamma_{M}^{\cyc}$) be the Galois group of the anticyclotomic $\ZZ_{p}^{2}$-extension $M_{\infty}^{\ac}$ (resp. the cyclotomic $\ZZ_p$-extension $M_{\infty}^{\cyc}$). Let $M_{\infty}:= M_{\infty}^{\cyc} M_{\infty}^{\ac}$ be their compositum, and put $\Gamma_{M} := \Gal(M_{\infty}/M)$ and $\Lambda_{M} = \calO \lrbracket{\Gamma_{M}}$.

Let $\II$ be the branch of the parallel-weight ordinary Hilbert Hida family through $f_F$. Let
\[
\calL_p(f_F/M)\in \Lambda_{M} \otimes \QQ_p
\]
be the associated three-variable $p$-adic $L$-function of \cite[Section 7.3]{Wan2015}, which interpolates the central $L$-values $L^{\alg}(f_{F}/M \otimes \chi, k-1)$ as $\chi$ varies over finite-order characters of $\Gamma_{M}$ (see \cite[Theorem 82 (i)]{Wan2015}). By the Gorenstein property verified above, \cite[Theorem 82 (ii)]{Wan2015} shows that $\calL_{p}(f_F/M) \in \Lambda_{M}$, i.e. that it is integral.

\begin{proposition}[Wan's Iwasawa main conjecture]
\label{prop:Wan-specialized}
Assume \ref{BC:K}--\ref{BC:p5}.  Then the ordinary Selmer module over $M_\infty$ satisfies
\begin{equation}\label{eq:Wan-specialized}
 \bigl(\calL_p(f_F/M)\bigr)
       \supseteq
 \chari_{\Lambda_M}X_{\ordi}(f_F/M_\infty)
       \qquad\text{in }\Lambda_M.
\end{equation}
\end{proposition}
The following proof is essentially that of \cite[Theorem~3.2.1]{BCS}.

\begin{proof}
The result \cite[Theorem~3]{Wan2015} applies to the parallel-weight $\II$-adic ordinary family $\bff_{F}$ when one arithmetic specialization is an ordinary cuspidal form of even parallel weight $\kappa_0 \geq 2$. Its hypotheses were verified above, and the theorem yields
\[
(\calL_p(\bff_F/M)) \supseteq \chari_{\II \lrbracket{\Gamma_M}} X_{\ordi}(\bff_F/M_\infty).
\]
Specializing the big Selmer group $X_{\ordi}(\bff_F/M_\infty)$ at the arithmetic prime of $\II$ corresponding to the ordinary stabilization of $f_F$, \cite[Proposition~15]{Wan2015} (as a Hilbert modular form analogue of \cite[(3.5)]{SkinnerUrban}) shows that it is exactly $X_{\ordi}(f_F/M_\infty)$. The specialization of $\calL_p(\bff_F/M)$ is $\calL_p(f_F/M)$ by definition. The assertion then follows.
\end{proof}

\subsubsection{Projection from $M$ to $K$}

Put $\widetilde K_\infty=FK_\infty$, which is a $\ZZ_p^{2}$-extension of $M$, and let $\widetilde\Lambda_K =\calO \lrbracket{\Gal(\widetilde K_\infty/M)}$. We identify $\widetilde{\Lambda_{K}}$ with $\Lambda_{K}$ and denote by
\[
\pi_K:\Lambda_M\longrightarrow\widetilde\Lambda_K
\]
the natural projection induced by the restriction $\Gal(M_{\infty}/M) \twoheadrightarrow \Gal(\widetilde{K}_{\infty}/M)$; set $I_K=\ker(\pi_K)$.

\begin{lemma} \label{lem:two-variable-torsion}
Assume that \eqref{eq:irred} holds. For $g=f$ and $g=f^F$, the module $X_{\ordi}(g/K_\infty)$ is a torsion $\Lambda_K$-module. As a result, $X_{\Gr}(g/K_{\infty})$ is a torsion $\Lambda_K^{\ur}$-module.
\end{lemma}

\begin{proof}
Choose a finite set $S$ containing $p$, every prime at which $T$ is ramified, and the primes ramified in $K/\QQ$, and let $X_{\ord}^{S}(g/K_{\infty})$ be the \emph{$S$-imprimitive} ordinary Selmer group. Combining the control theorem in \cite[Proposition 3.9]{SkinnerUrban} and Shapiro's lemma in \cite[Proposition 3.6]{SkinnerUrban}, we obtain the following isomorphism:
\[
\dfrac{X_{\ord}^{S}(g/K_{\infty})}{I_{\ac} X_{\ord}^{S}(g/K_{\infty})} \simeq X_{\ord}^{S}(g/\QQ_{\infty}) \oplus X_{\ord}^{\Sigma}(g \otimes \eta_{K} / \QQ_{\infty}).
\]
The result \cite[Theorem 17.4]{Kato} (see Theorem \ref{thm:kato_17_4}(1) for more details) shows that both $X_{\ord}(g/\QQ_{\infty})$ and $X_{\ord}(g \otimes \eta_{K} /\QQ_{\infty})$ are torsion over $\Lambda_{\cyc}$, and hence the same holds for $X_{\ord}^{S}(g/\QQ_{\infty})$ and $X_{\ord}^{S}(g \otimes \eta_{K} /\QQ_{\infty})$ (see \cite[Corollary 3.12]{SkinnerUrban}). This implies that the quotient $X_{\ord}^{S}(g/K_{\infty}) / I_{\ac} X_{\ord}^{S}(g/K_{\infty})$ is torsion over $\Lambda_{\cyc}$. It then follows that $X_{\ord}^{S}(g/K_{\infty})$ is torsion over $\Lambda_{K}$, and therefore its quotient $X_{\ord}(g/K_{\infty})$ is torsion over $\Lambda_{K}$ as well. The torsionness of $X_{\Gr}(g/K_{\infty})$ then follows from Theorem \ref{thm:equiv_IMC} together with Proposition \ref{prop:rank-one-nonvanishing}.
\end{proof}

The following lemma is proved as in \cite[Lemma 5.1.1]{BCS}; it is essentially Shapiro's lemma for ordinary Selmer groups (see \cite[Propositions~3.6--3.7]{SkinnerUrban}).

\begin{lemma}\label{lem:algebraic-M-to-K}
We have the divisibility \begin{equation}\label{eq:algebraic-M-to-K}
 \pi_K \left( \chari_{\Lambda_M}X_{\ordi}(f_F/M_\infty) \right)
 \supseteq
 \chari_{\Lambda_K}X_{\ordi}(f/K_\infty) \chari_{\Lambda_K}X_{\ordi}(f^F/K_\infty).
\end{equation}
\end{lemma}

The following is the analytic analogue of the factorization in Lemma \ref{lem:algebraic-M-to-K} above. It is \cite[Lemma 5.1.2]{BCS} which follows directly by comparing interpolation formulas (see \cite[Proposition 84]{Wan2015}).

\begin{lemma}
\label{lem:analytic-M-to-K}
There is an equality of ideals
\begin{equation}\label{eq:analytic-M-to-K}
 \bigl(\pi_K(\calL_p(f_F/M))\bigr)
  =\bigl(\calL_p^{\PR}(f/K)\calL_p^{\PR}(f^F/K)\bigr)
       \quad\text{in }\Lambda_K[1/\varpi].
\end{equation}
\end{lemma}

\subsubsection{Product of Iwasawa main conjectures}

We are now ready to prove the following higher-weight analogue of \cite[Proposition 5.2.1]{BCS}.

\begin{theorem} \label{thm:BCS-521-higher}
Let $f$ have even weight $k=2r\ge2$, trivial nebentype, and good ordinary reduction at $p\ge5$, and suppose that \eqref{eq:irred} holds. Let $K$ and $F$ satisfy \ref{BC:K}--\ref{BC:p5}. Then
\begin{align}
 &(\calL_p^{\PR}(f/K)\calL_p^{\PR}(f^F/K))
  \supseteq
  \chari_{\Lambda_K}X_{\ordi}(f/K_\infty)
  \chari_{\Lambda_K}X_{\ordi}(f^F/K_\infty) \quad \text{in } \Lambda_{K},
  \label{eq:BCS-standard-integral}\\
 &(\calL_p^{\Gr}(f/K)\calL_p^{\Gr}(f^F/K))
  \supseteq
  \chari_{\Lambda_K^{\ur}}X_{\Gr}^{\ur}(f/K_\infty)
  \chari_{\Lambda_K^{\ur}}X_{\Gr}^{\ur}(f^F/K_\infty) \quad \text{in } \Lambda_{K}^{\ur}.
  \label{eq:BCS-greenberg-integral}
\end{align}
\end{theorem}

\begin{proof}
Combining Proposition~\ref{prop:Wan-specialized} with Lemmas~\ref{lem:algebraic-M-to-K} and \ref{lem:analytic-M-to-K} gives
\begin{equation}\label{eq:standard-product-rational}
 \bigl(\calL_p^{\PR}(f/K)\calL_p^{\PR}(f^F/K)\bigr) \supseteq \chari_{\Lambda_{K}} X_{\ordi}(f/K_\infty)\, \chari_{\Lambda_{K}} X_{\ordi}(f^F/K_\infty) \quad\text{in }\Lambda_K[1/\varpi].
\end{equation}
Applying Theorem~\ref{thm:equiv_IMC}(2), together with Proposition \ref{prop:rank-one-nonvanishing}, over $\Lambda_K^{\ur}[1/\varpi]$, we obtain \eqref{eq:BCS-greenberg-integral} after inverting $\varpi$, namely
\begin{equation} \label{eq:BCS-greenberg-non-integral}
(\calL_p^{\Gr}(f/K)\calL_p^{\Gr}(f^F/K))
  \supseteq
  \chari_{\Lambda_K^{\ur}}X_{\Gr}^{\ur}(f/K_\infty)
  \chari_{\Lambda_K^{\ur}}X_{\Gr}^{\ur}(f^F/K_\infty) \quad \text{in } \Lambda_{K}^{\ur}[1/\varpi].    
\end{equation}

Condition \ref{BC:K} is the classical Heegner hypothesis for both $f$ and $f^F$: primes dividing $D_F$ split in $K$, so all new primes in the level of the quadratic twist also split in $K$. Condition~\ref{BC:irrM} gives absolute irreducibility over $K$ for both $f$ and $f^F$. Theorem \ref{thm:mu-vanishing} therefore gives $\mu(\calL_p^{\Gr}(f/K)) = \mu(\calL_p^{\Gr}(f^F/K)) = 0$. The product $\calL_p^{\Gr}(f/K) \calL_p^{\Gr}(f^F/K)$ also has $\mu$-invariant zero. Therefore the inclusion \eqref{eq:BCS-greenberg-non-integral} is valid integrally in $\Lambda_{K}^{\ur}$; that is, \eqref{eq:BCS-greenberg-integral} holds.

A second application of Theorem \ref{thm:equiv_IMC}(1) (together with Proposition \ref{prop:rank-one-nonvanishing}), now over integral Iwasawa algebras, gives \eqref{eq:BCS-standard-integral}. The theorem is therefore proved.
\end{proof}

\subsection{Applications to cyclotomic main conjectures} \label{sec:application_cmc}
We impose the following assumption for the remainder of the article:
\begin{equation} \label{eq:irred_Q} \tag{irred$_{\QQ}$}
    \text{The residual $\Gal_{\QQ}$-representation $\barV$ is absolutely irreducible.}
\end{equation}
Note that this condition is weaker than \eqref{eq:irred}.

It is shown in \cite[Lemma 5.2.3]{BCS} that the pair of quadratic fields $K$ and $F$ satisfying \ref{BC:K}--\ref{BC:p5} exists under assumption \eqref{eq:irred_Q}. From this point onward, we \emph{fix} such a pair of quadratic fields $K$ and $F$.

Recall that $\Lambda=\calO \lrbracket{\Gamma}$ denotes the Iwasawa algebra for the cyclotomic $\ZZ_p$-extension $\QQ_{\infty}/\QQ$. For 
\[
\delta\in\{1,\eta_K,\eta_F,\eta_{FK}\},
\]
we write $f^\delta=f\otimes\delta$ for the twist of $f$ by the character $\delta$. By $f^{\eta_{FK}}$ we mean the quadratic twist $(f^{F})^{K}$ of $f^{F}$ over $K$. The residual representation of every $f^\delta$ is absolutely irreducible, since it remains irreducible after restriction to $G_M$. We first recall Kato's divisibility in the cyclotomic Iwasawa main conjecture for $f$ over $\QQ$.

\begin{theorem}[Kato, {\cite[Theorem~17.4]{Kato}}] \label{thm:kato_17_4}  For $\delta \in \{1,\eta_K,\eta_F,\eta_{FK}\}$, the following statements hold.
\begin{enumerate}[label = \rm (\arabic*)]
    \item The Selmer groups $X_\ordi(f^{\delta}/\QQ_{\infty})$ are torsion $\Lambda$-modules.
    \item We have the divisibility $(\calL_{p}(f^{\delta}/\QQ)) \subseteq \chari_\Lambda X_\ordi(f^{\delta}/\QQ_{\infty})$ holding in $\Lambda[1/\varpi]$.
    \item Suppose moreover that 
    \begin{equation} \label{eq:Kato-image} \tag{im}
\text{There is an element $\sigma\in G_{\QQ(\mu_{p^\infty})}$ such that $T/(\sigma-1)T\simeq\calO$}.
\end{equation}
Then the divisibility $(\calL_{p}(f^{\delta}/\QQ)) \subseteq \chari_\Lambda X_\ordi(f^{\delta}/\QQ_{\infty})$ holds in $\Lambda$.
\end{enumerate}
\end{theorem}

\begin{proof}
    This is proved in \cite[Theorem 17.4]{Kato}; we provide some additional details. Under our choice of quadratic fields $K$ and $F$, the prime $p$ does not divide the conductor of $\delta$, so $f^{\delta}$ still has good reduction at $p$. Since $a_p(f^{\delta}) = \delta(p) a_{p}(f)$ and $\delta(p) \in \{\pm 1\}$, the form $f^{\delta}$ remains $p$-ordinary. Applying \cite[Theorem 17.4 (1)(2)]{Kato} gives parts (1) and (2).
    
    It suffices to show that \eqref{eq:Kato-image} is stable under quadratic twists. In other words, assuming \eqref{eq:Kato-image}, for every quadratic character $\delta$ there exists $\sigma_\delta\in G_{\QQ(\mu_{p^\infty})}$ such that $(T\otimes\delta)/(\sigma_\delta-1)(T\otimes\delta)\simeq\calO$.

    Put $A=\rho_T(\sigma)$. Since $\sigma$ fixes $\QQ(\mu_{p^\infty})$ and $f$ has trivial nebentype, $\det(A)= \epsilon_{\cyc}^{k-1}(\sigma) = 1$. Condition \eqref{eq:Kato-image} is equivalent to $\coker(A-1) = \calO$, which implies $\det(A-1) = 0$. For the $(2 \times 2)$-matrix $A-1$, one checks directly that
    \[
    \det(A-1) = 1 - \tr(A) + \det(A).
    \]
    This gives $\tr(A) = 2$. Moreover, one checks directly that
    \[
    \det(A+1) = 1 + \tr(A) + \det(A) = 4 \in \calO^{\times}, 
    \]
    since $p$ is an odd prime. This implies that $A+1 \in \GL_2(\calO)$. In other words, the linear operator $\sigma+1$ on $T$ is invertible. We consider the two cases:

    \underline{(1) $\delta(\sigma) = 1$}. In this case, we simply take $\sigma_{\delta} = \sigma$. Then
    \[
    \rho_{T^{\delta}}(\sigma_{\delta}) = \delta(\sigma) \sigma = \sigma,
    \]
    and hence $T^{\delta}/(\sigma_{\delta}-1)T^{\delta} \simeq T/(\sigma - 1) \simeq \calO$, as desired.

    \underline{(2) $\delta(\sigma) = -1$}. In this case, we take $\sigma_{\delta} = \sigma^{2}$. Then $\delta(\sigma_{\delta}) = \delta(\sigma)^{2} = 1$. Hence 
    \[
    \rho_{T^{\delta}}(\sigma_{\delta}) = \delta(\sigma^{2}) \sigma^{2} = \sigma^{2}.
    \]
    Therefore,
    \[
    T^{\delta}/(\sigma_{\delta}-1)T^{\delta} \simeq T/(\sigma^{2} - 1)T \simeq T/(\sigma-1)(\sigma+1)T \simeq T/(\sigma-1)T \simeq \calO
    \]
    where we have used the fact that $\sigma +1$ is an invertible linear operator on $T$. Part (3) of the theorem then follows from \cite[Theorem 17.4 (3)]{Kato}.
    \end{proof}

We are now ready to prove the following main theorem of this article.
\begin{theorem}[{Theorem \ref{main:A}}]
\label{thm:higher-weight-cyclotomic-MC}
Let $f \in S_{k}(\Gamma_{0}(N))$ be a normalized newform of even weight $k =2r \geq 2$ and trivial nebentype such that \eqref{eq:irred_Q} holds. If $p \geq 5$, $p \nmid N$, and $f$ is $p$-ordinary, then the following statements hold.
\begin{enumerate}[label = \rm (\arabic*)]
    \item The Selmer group $X_{\ord}(f/\QQ_{\infty})$ is $\Lambda$-torsion, and the equality
    \[
    (\calL_{p}(f/\QQ)) = \chari_\Lambda X_\ordi(f/\QQ_{\infty})
    \]
    holds in $\Lambda[1/\varpi]$.
    \item If in addition the condition \eqref{eq:Kato-image} holds, then the equality holds in $\Lambda$.
\end{enumerate}
\end{theorem}

The following proposition is a crucial intermediate step.

\begin{proposition}
\label{prop:cyclotomic-product-bound}
Under the hypotheses of Theorem~\ref{thm:higher-weight-cyclotomic-MC}, we have the following divisibility in $\Lambda$:
\begin{multline} \label{eq:fourfold-upper}
    (\calL_p(f/\QQ) \cdot \calL_p(f^{K}/\QQ) \cdot \calL_p(f^{F}/\QQ) \cdot \calL_p(f^{FK}/\QQ))
 \\ \supseteq
 \chari_\Lambda(X_{\ordi}(f/\QQ_{\infty})) \, \chari_\Lambda(X_{\ordi}(f^{K}/\QQ_{\infty})) \\ \cdot \chari_\Lambda(X_{\ordi}(f^{F}/\QQ_{\infty})) \, \chari_\Lambda(X_{\ordi}(f^{FK}/\QQ_{\infty})).
\end{multline}
\end{proposition}

\begin{proof}
Apply the cyclotomic quotient map $\pi_{\cyc}: \Gamma_{K} \twoheadrightarrow \Gamma_{K}^{\cyc}$ to \eqref{eq:BCS-standard-integral}. The cyclotomic descent theorem (Theorem \ref{thm:descent_theorem}(1)), applied first to $f$ and then to $f^F$, identifies the images of $\calL_{p}^{\PR}(f/K)$ and $\calL_{p}^{\PR}(f^{F}/K)$ with the product on the left-hand side of \eqref{eq:fourfold-upper}.

On the algebraic side, Shapiro's lemma for ordinary Selmer groups (see \cite[Propositions~3.6 and 3.9]{SkinnerUrban} and \cite[Corollary~3.8]{SkinnerUrban}) gives
\begin{align*}
 \pi_{\cyc}(\chari_{\Lambda_{K}} X_{\ordi}(f/K_\infty))
   &\supseteq\chari_\Lambda(X_{\ordi}(f/\QQ_{\infty}))\chari_\Lambda(X_{\ordi}(f^{K}/\QQ_{\infty})),\\
 \pi_{\cyc}(\chari_{\Lambda_{K}} X_{\ordi}(f^F/K_\infty)) &\supseteq \chari_{\Lambda} (X_{\ord}(f^{F}/\QQ_{\infty}))\chari_\Lambda(X_{\ordi}(f^{FK}/\QQ_{\infty})),
\end{align*}
upon the identification $\Lambda_{K}^{\cyc} \simeq \Lambda$. This proves \eqref{eq:fourfold-upper}.
\end{proof}

\begin{proof}[Proof of Theorem \ref{thm:higher-weight-cyclotomic-MC}]
By Theorem \ref{thm:kato_17_4}, we have the other side of the divisibility 
\[
(\calL_{p}(f^{\delta}/\QQ)) \subseteq \chari_\Lambda X_\ordi(f^{\delta}/\QQ_{\infty})
\]
when \eqref{eq:Kato-image} holds, and in $\Lambda[1/\varpi]$ otherwise. Any strict divisibility would contradict Proposition \ref{prop:cyclotomic-product-bound}. This concludes the proof.
\end{proof}